\documentclass[10pt,francais]{smfart}

\usepackage[T1]{fontenc}
\usepackage[english,francais]{babel}

\usepackage{amssymb,url,xspace,smfthm}
\usepackage{tikz}
\usepackage{tikz-cd}
\usepackage{amsmath}
\usepackage{amsfonts}
\usepackage{xcolor}

\usepackage{enumitem}
\usepackage[colorlinks=true]{hyperref} 
\hypersetup{urlcolor=blue,linkcolor=cyan,citecolor=black,colorlinks=true}
\usepackage[symbol=$\uparrow~$,numberlinked=true]{footnotebackref}

\usepackage{comment}

\newcommand{\Spec}{\operatorname{Spec}}

\newcommand{\Hom}{\operatorname{Hom}}

\newcommand{\End}{\operatorname{End}}

\newcommand{\id}{\operatorname{id}}

\newcommand{\Frac}{\operatorname{Frac}}

\newcommand{\Gal}{\operatorname{Gal}}
\newcommand{\Frob}{\operatorname{Frob}}

\newcommand{\Lie}{\operatorname{Lie}}

\newcommand{\bF}{\mathbb{F}}
\newcommand{\bG}{\textbf{G}}

\newcommand{\bN}{\mathbb{N}}

\newcommand{\bP}{\textbf{P}}
\newcommand{\bQ}{\mathbb{Q}}

\newcommand{\bZ}{\mathbb{Z}}

\newcommand{\cO}{\mathcal{O}}

\newcommand{\fa}{\mathfrak{a}}
\newcommand{\fb}{\mathfrak{b}}
\newcommand{\fc}{\mathfrak{c}}

\newcommand{\fj}{\mathfrak{j}}

\newcommand{\fm}{\mathfrak{m}}

\newcommand{\fp}{\mathfrak{p}}
\newcommand{\fP}{\mathfrak{P}}

\usepackage[a4paper, top=3.5cm, bottom=3.5cm, left=3.3cm, right=3.3cm]{geometry}

\title{Pour une définition commune des courbes elliptiques et modules de Drinfeld}
\date {Juin 2023}

\author{Quentin Gazda}
\address{Centre de Mathématiques Laurent Schwartz (CMLS), \'Ecole Polytechnique, Cour Vaneau F-91120 Palaiseau}
\email{quentin@gazda.fr}
\urladdr{https://quentin.gazda.fr}

\author{Damien Junger}
\address{Mathematisches Institut, Universit\"at M\"unster, Fachbereich Mathematik und Informatik der Universität M\"unster, Orl\'eans-Ring 10, D-48149 M\"unster.}
\email{djunger@uni-muenster.de}

\begin{document}

\maketitle

\tableofcontents

\section{Introduction}
\subsection{Motivations}
Il est souvent énoncé que le module de Carlitz est à l'anneau des polynômes univariés sur un corps fini ce que le groupe multiplicatif est à l'anneau des entiers. Cette analogie s'étend au cas de "rang $2$", où les modules de Drinfeld jouent un rôle semblable à celui des courbes elliptiques. Ce travail est né dans l'optique de trouver une définition commune à ces objets, ne dépendant que de l'anneau des coefficients, et ainsi, relever cette analogie en une théorie commune. \\

Soit $A$ un anneau de Dedekind\footnote{Par convention, un \emph{anneau de Dedekind} est de dimension de Krull égale à $1$ (i.e. on exclura donc le cas des corps). } finiment engendré (sur $\bZ$). Soit également $\bG$ un schéma en $A$-modules sur un corps $L$; par \emph{schéma en $A$-modules}, nous entendons un foncteur
\[\bG:\mathbf{Alg}_L\longrightarrow \mathbf{Mod}_A\]
de la catégorie des $L$-algèbres vers celles des $A$-modules représenté par les points d'un $L$-schéma. Nous imposerons $\bG$ \emph{algébrique}\footnote{Cela signifie que le morphisme de schéma $X\to \Spec L$ associé est de type fini.}, \emph{connexe} et \emph{lisse}, ce qui signifie imposer les conditions éponymes au schéma sous-jacent. On peut considérer les modules de Tate $\ell$-adique de $\bG$ définis comme suit: pour $\ell$ un idéal maximal de $A$, $\operatorname{T}_\ell \bG$ est la limite inverse des points de $\ell^n$-torsion de $\bG$ sur une clôture séparable $L^s$ de $L$:
\[\operatorname{T}_\ell \bG:= \varprojlim \bG[\ell^n](L^s)=\varprojlim \left(\bG[\ell](L^s)\longleftarrow \bG[\ell^2](L^s)\longleftarrow \cdots \right).\]
C'est naturellement un module sur l'anneau $A_\ell$, obtenu en complétant $A$ le long de $\ell$. Le module de Tate $\ell$-adique ne dépend pas du choix de $L^s$ à isomorphismes de $A_\ell$-modules près.\\

Soit $r$ un entier positif non nul. Nous dirons de $\bG$ qu'il est \emph{de rang $r$} si, pour tout idéal maximal $\ell\subset A$, le module $\operatorname{T}_\ell \bG$ est libre de rang $r$ sur $A_\ell$ (cf. définition \ref{def:rang}).

\medskip
En introduction, et pour l'énoncé des résultats, nous donnons un nom à de tels objets:
\medskip
\begin{defi}\label{def:pre-module-elementaire}
Un schéma en $A$-modules algébrique connexe lisse sur $L$ de dimension $1$ et de rang $r$ est appelé \emph{pré-$A$-module élémentaire sur $L$}. On appellera $A$ \emph{l'anneau des coefficients}. 
\end{defi}

\medskip
\begin{exem}
Donnons ici quelques exemples de pré-modules élémentaires qui motivent notre définition, et renvoyons à la section \ref{sec:expo} pour les détails.
\begin{enumerate}[label=$-$]
\item Le groupe multiplicatif $\bG_m$ sur $\bQ$ en est l'exemple le plus simple avec $A=\bZ$ et de rang~$1$. On montrera d'ailleurs que les formes de $\bG_m$ sont les seuls pré-modules élémentaires pour ces paramètres (voir le théorème \ref{thm:classification} plus bas). 
\item Les courbes elliptiques sont également des exemples de pré-modules élémentaires de coefficient $\bZ$, cette fois-ci de rang $2$ . Les courbes elliptiques à multiplication complexe par $\cO_K$ -- où $K$ est quadratique imaginaire -- vu comme des pré-modules élémentaires de coefficient $\cO_K$ sont de rang $1$.
\item Soit $\bF=\bF_q$ un corps fini à $q$ éléments. Les \emph{modules de Drinfeld de rang $r$} sont également des exemples de pré-modules élémentaires et de rang $r$, de coefficient $A$ l'anneau des fonctions régulières sur une $\bF$-courbe projective lisse privée d'un point fermé. Rappelons que pour $L$ une extension finie de $\bF(C)$, un \emph{module de Drinfeld de rang $r$ sur $L$} est un foncteur
\[
E: \mathbf{Alg}_L\longrightarrow \mathbf{Mod}_A
\]
qui à une $L$-algèbre $R$ associe le $A$-module $E(R)$ construit ainsi : en tant qu'espace vectoriel sur $\bF$, $E(R)$ est $R$ lui-même, l'action de $a\in A$ est quant à elle déterminée par l'existence de coefficients $(a)_i\in L$ pour lesquels:
\[
\text{Pour~tout~}x\in R:\quad a\cdot x:=(a)_0x+(a)_1x^q+(a)_2x^{q^2}+\cdots+(a)_{rd}x^{q^{rd}}
\]
où $d=\deg(a)$ et $(a)_{rd}\neq 0$. L'association $a\mapsto (a)_0$ définit un morphisme d'anneaux $\delta_E:A\to L$ appelé \emph{morphisme caractéristique de $E$}. On dit que $E$ est \emph{générique} si $\delta_{E}$ coïncide à l'inclusion $A\subset L$ (cf. la sous-section \ref{subsec:drinfeld}).
\item Le \emph{module de Carlitz} est l'exemple le plus simple de module de Drinfeld générique pour $A=\bF[t]$, où $t$ agit par $t\cdot x:=tx+x^q$.
\end{enumerate}
\end{exem}
\medskip

Bien que distincts en apparence, les deux foncteurs $\bG_m$ et $\mathbf{C}$ jouent des rôles homologues en arithmétique lorsque l'on suit l'analogie $(\bZ,\bQ)\sim (\bF[t],\bF(t))$. Par exemple, les extensions finies abéliennes de $\bQ$ sont obtenues en adjoignant les éléments de torsion de $\bG_m(\bar{\bQ})$ (théorème de Kronecker-Weber); c'est également le cas des extensions finies abéliennes de $\bF(t)$ en adjoignant cette fois-ci la torsion du module de Carlitz\footnote{En adjoignant la torsion de $\mathbf{C}(\bF(\theta)^s)$ on n'obtiendrait en fait que l'extension abélienne maximale totalement ramifiée au point $\infty$ de $\bP^1_{\bF}$. Pour obtenir l'extension abélienne maximale, il faudrait également ajouter la torsion du module de Carlitz associé aux coefficients $\bF[1/t]$ par exemple.} (voir \cite{carlitz1,carlitz2}). En rang $2$, il est également d'usage de comparer les courbes elliptiques et les modules de Drinfeld de rang $2$.\\

On cherche alors une définition commune à $\bG_m$ et $\mathbf{C}$, plus généralement courbes elliptiques et modules de Drinfeld, ne dépendant que du corps global. Contrairement au cas de coefficient~$\bZ$, il existe pléthore de pré-modules élémentaires et de rang~$1$ de coefficient $\bF[t]$ qui ne sont pas des formes de $\mathbf{C}$. La notion de pré-module élémentaire est alors insuffisante pour obtenir la classification recherchée en caractéristique $p>0$, et une hypothèse additionnelle satisfaite simultanément par les objets listés ci-dessus est la bienvenue. Dans ce texte, nous étudierons les deux conditions suivantes indépendamment.\\

\paragraph*{Module élémentaire de type $(1)$} Soit $K$ le corps des fractions de $A$ et soit $L$ une extension finie de $K$. Soit $\bG$ un pré-$A$-module élémentaire sur $L$. Son espace tangent $\Lie_\bG(L)$ est un $L$-espace vectoriel de dimension $1$ muni par fonctorialité d'une structure de $A$-module qui commute à celle de $L$-espace vectoriel. C'est \emph{l'action tangentielle de $A$}. Comme $\End_L(\Lie_\bG(L))$ s'identifie canoniquement à $L$ comme anneau, on obtient un morphisme d'anneaux $\delta_\bG:A\to L$ que l'on appelle \emph{morphisme caractéristique de $\bG$}. On émet la définition suivante:
\medskip
\begin{defi}[cf. définition \ref{def:type1}]
Nous dirons de $\bG$ que c'est un module élémentaire de type $(1)$ si $\delta_\bG$ coïncide à l'inclusion $A\subset L$. 
\end{defi}
\medskip
\begin{rema}
Observons que la condition introduite est trivialement satisfaite lorsque ${A=\bZ}$ car $\bZ$ est initial. Par ailleurs, elle est vérifiée par tout module de Drinfeld générique, dont l'action tangentielle de $a\in A$ n'est autre que 
\[
\partial(x\mapsto a\cdot x):=\partial_x(ax+(a)_1x^q+\cdots )=a.
\]
\end{rema}
\medskip
\paragraph*{Module élémentaire de type $(2)$} Soit $\cO_L$ la clôture intégrale de $A$ dans $L$. Soit $\bG$ un pré-$A$-module élémentaire sur $L$. Nous dirons de $\bG$ que c'est un module élémentaire de type $(2)$ si la représentation Galoisienne $\operatorname{T}_\ell \bG$ est \emph{indépendante de $\ell$} dans le sens suivant:
\medskip
\begin{defi}[cf. définition \ref{def:type2}]
Nous dirons de $\bG$ que c'est un \emph{module élémentaire de type~$(2)$} s'il existe un ensemble fini $S$ d'idéaux maximaux de $\cO_L$ pour lequel, pour tout $\fP$ un idéal maximal de $\cO_L$ en dehors de $S$ et $\ell$ un idéal maximal de $A$ différent de $\fp:=\fP\cap A$,
\begin{enumerate}[label=$(\alph*)$]
    \item La représentation $\operatorname{T}_\ell \bG$ est non ramifiée en $\fP$; i.e. le groupe d'inertie $I_\fP\subset G_L$ en $\fP$ agit trivialement, et, 
    \item Le déterminant $s(\fP)\in A_\ell$ de l'action de $\Frob_\fP\in G_L/I_\fP$ sur $\operatorname{T}_\ell \bG$ appartient à $A$ et est indépendant de $\ell$.
\end{enumerate}
\end{defi}

\subsection{Présentation des résultats}
L'existence de pré-module élémentaires est très restrictive en l'anneau $A$. C'est ce que l'on montrera à travers le résultat suivant:
\medskip
\begin{theo}\label{thm:coefficient}
Supposons qu'il existe un pré-module élémentaire de coefficient $A$. Alors,
\begin{enumerate}[label=$(\roman*)$]
    \item\label{thm:coefficient-i} Si $A$ est de caractéristique $0$, soit $A=\bZ$, soit $A=\cO_K$ où $K$ est un corps quadratique imaginaire.
    \item\label{thm:coefficient-ii} Si $A$ est de caractéristique $p>0$, il existe une courbe projective lisse $(C,\cO_C)$ sur $\bF_p$ ainsi qu'un point fermé $\infty$ de $C$, tel que  $A=\cO_C(C\setminus\{\infty\})$.
\end{enumerate}
\end{theo}
\medskip
En particulier, l'existence d'un pré-module élémentaire force $K$ à être un corps global ayant au plus une place infinie (ce qui n'est restrictif que pour les corps de nombres). \\

On s'intéresse ensuite à la classification des modules élémentaires de type~$(1)$ et $(2)$. Notre résultat principal énonce que les modules élémentaires cités ci-dessus sont à peu de chose près les seuls:
\medskip
\begin{theo}\label{thm:classification}
Soit $\bG$ un module élémentaire de type $(1)$ ou $(2)$, de coefficient $A$ et de rang $r$. Alors:
\begin{enumerate}[label=$(\operatorname{\Roman*})$]
    \item\label{thm:classification-i} Si la caractéristique de $A$ est nulle,
    \begin{enumerate}[label=$-$]
        \item Soit $r=1$ et $A=\bZ$, auquel cas $\bG$ est une forme de $\bG_m$.
        \item Soit $r=1$ et $\bZ\subsetneq A$, auquel cas $A=\cO_K$ où $K$ est un corps quadratique imaginaire et $\bG$ est une courbe elliptique avec multiplication complexe par $\cO_K$.
        \item Soit $r=2$ auquel cas $A=\bZ$ et $\bG$ est une courbe elliptique.
    \end{enumerate}
    \item\label{thm:classification-ii} Si la caractéristique de $A$ est $p>0$, alors $A$ est comme dans le~théorème~\ref{thm:classification}.\ref{thm:coefficient-ii},~et
    \begin{enumerate}[label=$-$]
        \item Soit $\bG$ est de type $(1)$ auquel cas $\bG$ est un module de Drinfeld générique de coefficient $A$ et de rang $r$. 
        \item Soit de type $(2)$ auquel cas $\bG$ est un module de Drinfeld de caractéristique $\delta$, où $\delta:A\to L$ est l'élévation à une puissance $p$ème. 
    \end{enumerate}
\end{enumerate}
\end{theo}

\medskip
\begin{rema}\label{rem:presque-vrai}
La réciproque du théorème ci-dessus est \emph{presque vraie}, à ceci près qu'une courbe elliptique à multiplication complexe par $\cO_K$ n'est un module élémentaire de type $(1)$ que si son \emph{seul type} au sens de Shimura \cite{shimura} (ou \emph{morphisme caractéristique} sous notre terminologie) est l'inclusion $\cO_K\subset L$ (c.f. proposition \ref{prop:courbe-elliptique-cm}). 
\end{rema}
\medskip

\medskip
\begin{rema}
Si $\bG$ est un module élémentaire sur un corps $L$ de caractéristique $p>0$, et $\Frob_p:L\to L$ désigne l'élévation à la puissance $p$, alors $\bG':=\Frob_p^*\bG$ est encore un module élémentaire. Cela vient du fait que le groupe de Galois ne voit pas les racines $p$èmes, d'où $\operatorname{T}_\ell \bG\cong \operatorname{T}_\ell \bG'$ en tant que représentations de $G_L$. Cela explique pourquoi, en \ref{thm:classification-ii}, on peut avoir des modules de Drinfeld non génériques (i.e. dont le morphisme caractéristique diffère de l'inclusion).
\end{rema}
\medskip

\paragraph{Travail futur:} Tout au long de ce texte nous avons supposé $\bG$ de dimension $1$. Il serait grandement souhaitable de supprimer cette hypothèse et d'inclure dans cette étude le cas des modules d'Anderson d'un côté et des variétés semi-abéliennes de l'autre. Cependant, on trouve en dimension supérieure des schémas en $A$-modules dont les modules de Tate sont nuls (e.g. des groupes unipotents lorsque $K=\bQ$) et le projet devient sensiblement plus compliqué.

\subsection{Plan de l'article}
Nous établissons les définitions de modules élémentaires et présentons les principaux exemples en section \ref{sec:expo}. En section \ref{sec:generalité} on rappellera des résultats clés sur les groupes algébriques nécessaires à notre étude. Les preuves des théorèmes \ref{thm:coefficient} et \ref{thm:classification} sont très différentes en fonction de la caractéristique de $A$ (nulle ou positive). C'est pourquoi, en section \ref{sec:demo}, nous avons décidé de les traiter sous deux sous-sections distinctes.

\subsection{Remerciements}
Les deux auteurs souhaitent remercier chaleureusement le Mathematisches Forschungsinstitut Oberwolfach où les idées exposées dans ce papier sont nées.

\section{Modules élémentaires}\label{sec:expo}
Dans cette section, nous définissons les modules élémentaires de type $(1)$ et $(2)$ et présentons les exemples de tels objets. 

\subsection{Terminologie}
Soit $A$ un anneau de Dedekind et soit $L$ un corps qui est une $A$-algèbre. Fixons $L^s$ une clôture séparable de $L$ de groupe de Galois absolu noté $G_L$. Soit $\bG$ un schéma en $A$-modules sur $L$.\\

On rappelle que $\Lie_\bG$ est le foncteur de la catégorie des $L$-algèbres vers celle des $A$-modules qui, à une $L$-algèbre $R$, assigne
\[
\Lie_\bG(R):=\ker \bG\!\left(R[\varepsilon]/\varepsilon^2\xrightarrow{\varepsilon\mapsto 0} R\right).
\]
Le module $\Lie_\bG(R)$ est également un $L$-espace vectoriel où la multiplication par le scalaire $l\in L$ est déduite de l'endomorphisme $a+\varepsilon b\mapsto a+\varepsilon l b$ de l'anneau $R[\varepsilon]/\varepsilon^2$. En particulier, $\Lie_\bG(R)$ est naturellement un $A\otimes_{\bZ} L$-module. Si $\bG$ est lisse de dimension $d$, alors $\Lie_\bG(L)$ est de dimension $d$ sur $L$ \cite[cor. 1.23]{milne}. \\

Supposons $\bG$ lisse de dimension $d=1$. 
\medskip
\begin{defi}\label{def:morphisme-cara}
On appelle \emph{morphisme caractéristique de $\bG$} l'unique morphisme d'anneaux $\delta_\bG:A\to L$ pour lequel l'action de $a\in A$ sur $\Lie_\bG(L)$ coïncide avec la multiplication par le scalaire $\delta_{\bG}(a)\in L$. On appelle \emph{idéal caractéristique}, que l'on note $\fc_\bG$, le noyau de $\delta_\bG$.
\end{defi}
\medskip

Pour $\ell\subset A$ un idéal, le \emph{module de Tate $\ell$-adique} est défini comme la limite inverse des éléments de $\ell$-torsion:
\[
\operatorname{T}_\ell\bG:=\varprojlim_n \bG[\ell^n](L^s).
\]
En notant $A_\ell$ le complété de $A$ pour la topologie $\ell$-adique, $\operatorname{T}_\ell\bG$ est alors un $A_\ell$-module muni d'une action compatible du groupe $G_L$.
\medskip
\begin{defi}\label{def:rang}
Soit $r\geq 1$ un entier. Nous dirons de $\bG$ \emph{qu'il est de rang $r$} si, pour tout idéal maximal $\ell\subset A$ différent de l'idéal caractéristique $\fc$, le module $\operatorname{T}_\ell\bG$ est libre de rang $r$ sur $A_\ell$.
\end{defi}
\medskip
\begin{exem}\label{ex:classique}
Pour $A=\bZ$, le groupe multiplicatif $\bG_m$ sur $L$ a pour morphisme caractéristique la flèche $\bZ\to L$. L'idéal caractéristique de $\bG_m$ correspond alors à la caractéristique de $L$ ce qui motive la terminologie. Pour $\ell$ un premier différent de la caractéristique, les polynômes $X^{\ell^k}-1$, $k>0$, sont séparables et donc
\[
\operatorname{T}_\ell\bG_m\cong \varprojlim_{x\mapsto x^\ell} (L^s)^\times
\]
est libre de rang $1$ sur $\bZ_\ell$. Cela montre que $\bG_m$ est de rang $1$. Plus généralement, toute courbe elliptique sur $L$ est de rang $2$; nous renvoyons le lecteur à  \cite[III.7.1]{silverman1} pour ce résultat bien connu. De cette même référence, on déduit que pour $A=\cO_K$ où $K$ désigne une extension quadratique imaginaire de $\bQ$ et $\cO_K$ désigne son anneau d'entiers, toute courbe elliptique sur $L$ ayant multiplication complexe par $A$ est de rang $1$ vu comme schéma en $A$-modules sur $L$. 
\end{exem}
\medskip

Par la structure de $A$-module sur $\bG$, il y a un morphisme canonique
\[
\varphi:A \longrightarrow \End_{\operatorname{grp}/L}(\bG)
\]
de l'anneau $A$ vers celui des endomorphismes de $\bG$ vu comme schéma en groupes sur $L$. Voici un lemme facile mais utile à divers endroit de la classification.
\medskip
\begin{lemm}\label{lem:phi-injective}
Supposons $\bG$ de rang $r\geq 1$. Alors le noyau de $\varphi$ est inclus dans $\fc$. 
\end{lemm}
\medskip
\begin{proof}
Soit $a\in A\setminus \fc$. Comme $A$ est Dedekind, l'idéal $(a)$ de $A$ se décompose en produit d'idéaux premiers $\ell_1^{c_1}\cdots \ell_t^{c_t}$, les $\ell_i$ étant deux à deux premiers entre eux et non contenus dans $\fc$, et l'on a $\bG[a](L^s)\cong \bG[\ell_1^{c_1}](L^s)\times \cdots \times \bG[\ell_t^{c_t}](L^s)$ par les restes chinois. Par définition, le module $\bG[a](L^s)$ est isomorphe à $(A/a)^r$. Ainsi, puisque $(a)\neq (a^2)$, l'inclusion $\bG[a](L^s)\subset \bG[a^2](L^s)$ est stricte. Il en est alors de même pour l'inclusion $\bG[a](L^s)\subseteq \bG(L^s)$, et donc $\varphi(a)$ est non nul. 
\end{proof}

\subsection{Modules élémentaires de type $(1)$}
Supposons à présent que $L$ est une extension finie de $K$, le corps des fractions de $A$. Soit $\bG$ un schéma en $A$-modules sur $L$ que l'on suppose connexe lisse de dimension $1$ et de rang $r$.
\medskip
\begin{defi}\label{def:type1}
Nous dirons de $\bG$ que c'est un \emph{$A$-module élémentaire sur $L$ de type $(1)$} si le morphisme caractéristique de $\bG$ coincide à l'inclusion $A\subset L$. En particulier, l'idéal caractéristique de $\bG$ est nul. 
\end{defi}
\medskip
\begin{rema}\label{rem:Z}
Si $A=\bZ$ alors, comme $\bZ$ est initial, tout schéma en $A$-modules sur $L$ de rang $r\geq 1$ est un module élémentaire de type $(1)$.
\end{rema}
\medskip

En vue de cette remarque et des exemples \ref{ex:classique}, on obtient:
\medskip
\begin{prop}\label{prop:reciproque1}
Les formes du groupe multiplicatif et les courbes elliptiques sur $L$ sont des $\bZ$-modules élémentaires de type $(1)$.
\end{prop}
\medskip

\`A une variété abélienne $A$ de dimension $d$ définie sur un corps de nombres $k$ et à multiplication par un corps $K$, Shimura \cite{shimura} associe son \emph{type} qui est un ensemble de $d$ plongements $K\to k^{\operatorname{alg}}$. Si $K$ est une extension quadratique imaginaire de $\bQ$ d'anneaux d'entiers $\cO_K$ et $E$ une courbe elliptique à multiplication complexe par $K$, il suit des définitions que le type est de $E$ est le singleton formé du morphisme caractéristique de $E$. En particulier:
\medskip
\begin{prop}\label{prop:courbe-elliptique-cm}
La courbe elliptique $E$ est un $\cO_K$-module élémentaire de type $(1)$ si son seul type au sens de Shimura coincide à l'inclusion $\cO_K\subset L$.
\end{prop}
\medskip

\subsection{Modules élémentaires de type $(2)$}
Soit $\bG$ un schéma en $A$-modules algébriques sur $L$, connexe lisse de dimension $1$ et de rang $r\geq 1$. Pour définir les modules élémentaires de type $(2)$, on fait usage de la propriété d'indépendance en $\ell$ suivante:
\medskip
\begin{defi}\label{def:type2}
Nous dirons de $\bG$ que c'est un \emph{$A$-module élémentaire de type~$(2)$ sur $L$} s'il existe un ensemble fini $S$ d'idéaux maximaux de $\cO_L$ pour lequel, pour tout $\fP$ un idéal maximal de $\cO_L$ en dehors de $S$ et $\ell$ un idéal maximal de $A$ différent de $\fp:=\fP\cap A$,
\begin{enumerate}[label=$(\alph*)$]
    \item La représentation $\operatorname{T}_\ell \bG$ est non ramifiée en $\fP$; i.e. le groupe d'inertie $I_\fP\subset G_L$ en $\fP$ agit trivialement, et, 
    \item Le déterminant $s(\fP)\in A_\ell$ de l'action de $\Frob_\fP\in G_L/I_\fP$ sur $\operatorname{T}_\ell \bG$ appartient à $A$ et est indépendant de $\ell$.
\end{enumerate}
\end{defi}
\medskip
\begin{exem}\label{ex:type(2)}
Il est encore une fois bien connu que, pour $A=\bZ$, le groupe multiplicatif et les courbes elliptiques vérifient ces hypothèses. Lorsque $A=\cO_K$ où $K$ est une extension quadratique imaginaire de $\bQ$, c'est également le cas des courbes elliptiques à multiplication complexe par $\cO_K$. Nous renvoyons le lecteur à \cite[Prop. V.2.3]{silverman1}. 
\end{exem}

\subsection{Modules de Drinfeld}\label{subsec:drinfeld}
Terminons cette section par des rappels de la théorie des modules de Drinfeld et montrons que ces derniers sont des modules élémentaires de type $(1)$ et $(2)$. \\ 

Soit $p$ un nombre premier et $\bF_p$ le corps fini à $p$ éléments. Soit $(C,\cO_C)$ une courbe projective lisse sur $\bF_p$ et $\infty$ un point fermé de $C$. Soit $A$ l'anneau des fonctions de $C$ régulières en dehors de $\{\infty\}$; i.e. $A=\cO_C(C\setminus\{\infty\})$. C'est un anneau de Dedekind. Pour $a\in A$ non nul, on notera $\deg(a)$ le \emph{le degré de $a$}; i.e. la dimension (finie) de $A/(a)$ sur $\bF_p$. \\

Soit $L$ un corps de caractéristique $p$. Désignons par $L\{\tau\}$ l'anneau non commutatif des sommes finies $p(\tau)=\sum_{i}{c_i\tau^i}$, $c_i\in L$ pour $i\geq 0$, la multiplication étant déduite de $\tau c=c^q \tau$. On notera $\deg_\tau p(\tau)$ l'entier donné par le maximum des $i\geq 0$ tels que $c_i\neq 0$. Si $\bG_a$ désigne le groupe additif sur $L$, on a alors un morphisme d'anneaux $L\{\tau\}\to \End_{\text{grp}/L}(\bG_a)$ où $c\in L$ opère par homothétie $x\mapsto cx$ sur $\bG_a$ et où $\tau$ opère via la mise à la puissance $p$ème $x\mapsto x^p$. On vérifie que c'est un isomorphisme d'anneaux (e.g.  \cite[II.\S 3,4.4]{demazure-gabriel}).\\

\'Etant donné $G$ un schéma en $A$-modules sur $L$, alors le groupe $M(G):=\Hom_{\text{grp}/L}(G,\bG_a)$ est muni d'une structure de $L\{\tau\}$-module à gauche par pré-composition. $M(G)$ est également muni d'une structure de $A$-module, $A$  agissant sur $G$, qui commute à l'action de $L\{\tau\}$. Ces actions coïncident sur $\bF_p$ et donc $M(G)$ est canoniquement un $A\otimes_{\bF_p}L\{\tau\}$-module à gauche. Rappelons la définition d'un module de Drinfeld:
\medskip
\begin{defi}[$A$-module de Drinfeld]
Un schéma en $A$-modules $E$ sur $L$ est appelé un \emph{$A$-module de Drinfeld sur $L$ de rang $r$} s'il est isomorphe à $\bG_a$ en tant que schéma en groupes sur $L$ et que le $A\otimes_{\bF_p} L$-module $M(E)$ est localement libre de rang $r$.
\end{defi}
\medskip
Comme tout module de Drinfeld sur $L$ et un schéma en $A$-modules de dimension $1$, on peut lui associer un morphisme caractéristique $\delta_E$ (définition \ref{def:morphisme-cara}). On pourrait démontrer qu'il n'y pas de conflit avec notre définition de rang (définition \ref{def:rang}) en utilisant \cite[prop. 2.2]{drinfeld}. On le montrera plutôt en utilisant le $A$-motif associé à un module de Drinfeld.

\subsubsection*{$A$-Motif associé}
Soit $\sigma$ l'endomorphisme de $A$-algèbres de $A\otimes_{\bF_p} L$ qui agit comme l'élévation à la puissance $p$ sur $L$. Soit $\delta:A\to L$ un morphisme d'anneaux. La définition suivante est due à Anderson \cite{anderson}:
\medskip
\begin{defi}
Un {\it $A$-motif (effectif, abélien) de rang $r$ et de caractéristique $\delta$} est la donnée d'un module $M$ localement libre de rang $r$ sur $A\otimes_{\bF_p}L$ et d'un morphisme $\sigma$-linéaire $\tau_M:M\to M$ dont le conoyau est annihilé par une puissance de l'idéal 
\[
\fj_{\delta}:=\ker(A\otimes_{\bF_p}L\to L,~a\otimes b\mapsto \delta(a)b) \subset A\otimes_{\bF_p} L.
\]
\end{defi}
\medskip

\'Etant donné un module de Drinfeld $E$, on obtient un $A$-motif de même morphisme caractéristique et rang, le module sous-jacent est $M(E)$ et le morphisme $\tau_M$ étant obtenu par la post-composition par $\tau\in \End_{\text{grp}/L}(\bG_a)$ sur $M(E)$ (nous renvoyons à \cite[Thm. 3.5]{hartl} pour les détails de cette construction). 
\medskip
\begin{defi}
La donnée $\underline{M}(E)$ de $(M(E),\tau_M)$ est appelé \emph{le $A$-motif associé à $E$}. 
\end{defi}
\medskip
D'après le théorème \ref{thm:equiv-cat} plus bas, le $A$-motif $\underline{M}(E)$ détermine le module de Drinfeld $E$ (bien que tous $A$-motifs ne soit pas nécessairement de la forme $\underline{M}(E)$). On retrouve $E$ par la formule:
\begin{equation}\label{eq:formule}
E:\mathbf{Alg}_L \longrightarrow \mathbf{Mod}_A, \quad R\longmapsto \Hom_{L\{\tau\}}(M(E),R).
\end{equation}
Soit $L^s$ une cl\^oture séparable de $L$ de groupe de Galois absolu $G_L$. Notons $\operatorname{T}_{\ell}E$ le module de Tate $\ell$-adique de $E$ relatif à $L^s$ et l'idéal maximal $\ell\subset A$. En conséquence de cette description, on a un isomorphisme de représentation $A_{\ell}$-linéaire de $G_k$:
\begin{equation}\label{eq:equiv-iso}
\operatorname{T}_{\ell}E \cong \Hom_{L\{\tau\}}\left(M(E)^{\wedge}_{\ell_L},L^s\right).
\end{equation}
où $M(E)^{\wedge}_{\ell_L}$ désigne la complétion de $M(E)$ pour la topologie $\ell_L$-adique, où $\ell_L$ désigne l'idéal $\ell\otimes_{\bF_p}L\subset A\otimes_{\bF_p}L$. En particulier, on en déduit:
\medskip
\begin{prop}\label{prop:rang}
Si $\ell$ est différent de $\fc=\ker \delta_E$, alors $\operatorname{T}_{\ell}E$ est libre de rang $r$ sur $A_\ell$. 
\end{prop}
\medskip
\begin{proof}
C'est un résultat bien connu dont on se borne à esquisser la preuve. Soit $n\geq 1$. Si $\ell\neq \fc$, alors l'application $p$-linéaire des $L$-espaces de dimension finie $M(E)/\ell_L^n\to M(E)/\ell_L^n$ déduite de $\tau$ est semi-simple (cf. \cite{katz}). En particulier, on déduit du théorème d'isogénie de Lang (e.g. prop. 1.1 de \emph{loc.\,cit.}) que la flèche de multiplication
\[
(M(E)/\ell_L^n\otimes_L L^s)^\tau\otimes_{\bF_p} L^s \longrightarrow M(E)/\ell_L^n\otimes_L L^s
\]
est un isomorphisme. En particulier $(M(E)/\ell_L^n\otimes_L L^s)^\tau$ est un $A/\ell^n$-module libre de rang $r$. De plus, on déduit de \eqref{eq:formule} des isomorphismes de $A/\ell^n$-modules:
\begin{align*}
E[\ell^n](L^s)&\cong \Hom_{L\{\tau\}}(M(E)/\ell_L^n,L^s)\cong \Hom_{L^s\{\tau\}}(M(E)/\ell_L^n\otimes _L L^s,L^s) \\
&\cong \Hom_{\bF_p}((M(E)/\ell_L^n\otimes _L L^s)^\tau,\bF_p).
\end{align*}
Ainsi, $E[\ell^n](L^s)$ est libre de rang $r$ sur $A/\ell^n$. Puisque l'on peut, à chaque étape, choisir des base compatibles, on obtient que $\operatorname{T}_\ell E$ est libre de rang $r$ en passant à la limite.
\end{proof}

\subsubsection*{Modules de Drinfeld et module élémentaires}
On suppose que $L$ est une extension finie de $K$, le corps de fractions de $A$. Soit $E$ un $A$-module de Drinfeld sur $L$ de rang $r$ et soit $\delta=\delta_E:A\to L$ son morphisme caractéristique. 
\medskip
\begin{defi}\label{def:generique}
Un $A$-module de Drinfeld $E$ sur $L$ est appelé \emph{générique} (ou \emph{de caractéristique générique}) si $\delta_E$ coincide avec l'inclusion $A\subset L$.
\end{defi}
\medskip
Alors:
\medskip
\begin{prop}\label{prop:drinfeld-elementaire}
Soit $E$ un $A$-module de Drinfeld sur $L$ générique. En tant que schéma en $A$-modules, $E$ est un module élémentaire de rang $r$ de type $(1)$ et $(2)$.
\end{prop} 
\medskip
\begin{proof}
Que $E$ soit connexe et lisse découle de son isomorphisme avec $\bG_a$; qu'il soit de rang $r$ selon la définition \ref{def:rang} découle de la proposition \ref{prop:rang}. Puisqu'il est générique, il est élémentaire de type $(1)$. Qu'il soit de type $(2)$ se déduit de \cite[cor. 3.2.4]{gossL}.
\end{proof}

Soit $G$ un schéma en $A$-modules sur $L$. Soit $\Frob_p:L\to L$, $x\mapsto x^p$ le $p$-Frobenius et désignons par $\Frob_p^*G$ le schéma en $A$-modules sur $L$ donné par $R\mapsto G(R^{(1)})$ où $R^{(1)}$ est la $L$-algèbre égale à $R$ en tant qu'anneaux et où la multiplication par $L$ agit à travers $\Frob_p$. Si $G$ est algébrique (resp. connexe ou lisse) il en est de même pour $\Frob_p^*G$. On en déduit:
\medskip
\begin{coro}
Soit $E$ un $A$-module de Drinfeld sur $L$ dont le morphisme caractéristique coïncide avec l'élévation à une puissance $p$ième. Alors $E$ un module élémentaire de type $(2)$. 
\end{coro}
\medskip
\begin{proof}
Soit $k\geq 0$ tel que $\delta=\Frob_p^k$. On a $E\times_L L^{1/p^k}\cong (\Frob_p^k)^* E_0$ où $L^{1/p^k}$ est le corps obtenu en adjoignant les racines $p^k$èmes des éléments de $L$ et où $E_0$ est un module de Drinfeld sur $L^{1/p^k}$ de morphisme caractéristique l'inclusion $A\to L^{1/p^k}$. En tant que schéma en $A$-modules, $E_0$ est un module élémentaire de type $(2)$ d'après la proposition \ref{prop:drinfeld-elementaire} et donc $E\times_L L^{1/p^k}$ est aussi élémentaire de type $(2)$. Comme $\Frob_p$ commute à l'action de Galois, il suit des définitions que $E$ est élémentaire de type $(2)$.
\end{proof}

On termine cette section par d'avantage de généralités sur le déterminant des modules de Drinfeld qui nous seront utiles en section \ref{sec:demo} pour se ramener au cas de rang $1$.

\subsubsection*{Déterminant d'un module de Drinfeld}
Soit $E$ un module de Drinfeld de rang $r$ et de caractéristique $\delta_E$. Soit $\underline{M}(E)=(M(E),\tau_M)$ son $A$-motif. Notons $D$ la puissance extérieure maximale $\bigwedge^{\operatorname{max}} M(E)=\bigwedge^{r} M(E)$ prise en tant que $A\otimes_{\bF_p} L$-module. $D$ est alors localement libre de rang $1$. Le morphisme $\tau_M$, en agissant diagonalement sur $D$, induit un morphisme $\sigma$-linéaire
\[
\tau_D:D \longrightarrow D, \quad m_1\wedge \cdots \wedge m_r\longmapsto \tau(m_1)\wedge \cdots \wedge \tau(m_r).
\]
Il est facile de vérifier que la donnée de $\underline{D}=(D,\tau_D)$ est encore un $A$-motif de caractéristique $\delta_E$, cette-fois-ci de rang $1$. On doit le résultat suivant à Drinfeld (voir \cite[\S 0]{anderson}):
\medskip
\begin{theo}\label{thm:det}
Le $A$-motif $\underline{D}$ provient d'un $A$-module de Drinfeld de rang $1$ et de caractéristique $\delta_E$. Ce module de Drinfeld est unique à isomorphisme près, et on le note $\det E$.
\end{theo}
\medskip
D'après \eqref{eq:equiv-iso}, on a un isomorphisme de représentations $A_\ell$-linéaires de $G_L$:
\begin{equation}\label{eq:Tate-det}
\operatorname{T}_{\ell}(\det E) \cong \bigwedge^{r}_{A_\ell}\operatorname{T}_{\ell}E.
\end{equation}
L'identité ci-dessus nous permettra de restreindre notre preuve au cas de rang~$1$. 

\section{Généralités sur les groupes algébriques et schémas en modules}\label{sec:generalité}
Dans cette section, nous rappelons quelques résultats classiques de la théorie des groupes algébriques qui nous serviront dans notre étude. Nous utiliserons la pleine puissance de deux théorèmes : celui de Barsotti-Chevalley (théorème \ref{thm:barsotti-chevalley}) ainsi que la classification des groupes algébriques annulés par décalage en caractéristique non nulle (théorème \ref{thm:equiv-cat}). 

\subsection{Théorème de Barsotti-Chevalley et conséquence}
Le théorème qui suit joue un rôle majeur dans notre classification (voir \cite[thm. 10.5]{milne} pour une démonstration):
\medskip
\begin{theo}[Barsotti-Chevalley]\label{thm:barsotti-chevalley}
Tout groupe algébrique connexe $G$ sur un corps parfait s'insère dans une suite exacte de groupes algébriques:
\begin{equation}\label{eq:barsotti-chevalley}
1\longrightarrow H\longrightarrow G\longrightarrow E\longrightarrow 1 
\end{equation}
où $E$ est une variété abélienne et $H\subset G$ est un sous-groupe algébrique normal connexe et affine.
\end{theo}
\medskip
Soit $A$ un anneau commutatif unitaire. Lorsque $G$ est équipé d'une structure de schéma en $A$-modules, on peut en dire plus:
\medskip
\begin{prop}\label{prop:barsotti-chevalley-A}
Soit $G$ un schéma en $A$-modules algébrique et connexe sur un corps parfait. Alors chaque  terme ($H$ ou $E$) de  la suite \eqref{eq:barsotti-chevalley} peut-être canoniquement muni d'une structure  de schéma en $A$-modules.
\end{prop}
\medskip
\begin{proof}
\'Etant donné $a\in A$, la structure de schéma en $A$-modules de $G$ produit un endomorphisme $\varphi(a)$ du groupe algébrique $G$. La composition 
\[
H\longrightarrow G\xrightarrow{\varphi(a)} G \longrightarrow E
\]
est un morphisme de groupes algébriques entre un groupe linéaire et une variété abélienne, et est donc nulle d'après \cite[lem. 2.3]{conrad}. L'exactitude de \eqref{eq:barsotti-chevalley} signifie que $\varphi(a)$ se factorise par $H\to G$ en un unique morphisme $\varphi_{H}(a):H\to H$. Puisque la catégorie des groupes algébriques commutatifs est abélienne, il existe un unique morphisme $\varphi_{E}(a):E\to E$ tel que \eqref{eq:barsotti-chevalley} se complète en un diagramme commutatif :
\begin{equation}
\begin{tikzcd}
1 \arrow[r] & H \arrow[r] \arrow[d,"\varphi_{H}(a)"] & G \arrow[r]\arrow[d,"\varphi(a)"] & E \arrow[r]\arrow[d,"\varphi_{E}(a)"] & 1 \\
1 \arrow[r] & H \arrow[r] & G \arrow[r] & E \arrow[r] & 1.
\end{tikzcd}\nonumber
\end{equation}
Par unicité, il est facile de voir que la donnée des $(\varphi_{H}(a))_{a\in A}$ et des $(\varphi_{E}(a))_{a\in A}$ enrichit $H$ et $E$ respectivement de structures de schéma en $A$-modules. 
\end{proof}

\subsection{Groupes unipotents en caractéristique non nulle}
Rappelons qu'un groupe algébrique est dit \emph{unipotent} si et seulement s'il admet une série centrale normale dont les quotients successifs sont des fermés de $\bG_a$.
\medskip
\begin{rema}
Cette définition est équivalente à l'existence de vecteurs invariants dans les représentations fidèles donnée usuellement dans la littérature (cf. \cite[prop. 15.23]{milne}).
\end{rema}
\medskip
Fixons $k$ un corps de caractéristique $p>0$ ainsi que $k^{\operatorname{perf}}$ sa perfection. Soit $G$ un groupe algébrique commutatif affine sur $k$ de décalage\footnote{Ou, en allemand, \emph{Verschiebung}.} $V_G$. \\

Dans la classe des groupes algébriques unipotents, nous distinguons les fermés de $\bG_a^d$ qui sont classés par le résultat suivant (cf. \cite[Thm. 6.6]{demazure-gabriel}).
\medskip
\begin{theo}\label{thm:caract-unipotent}
$G$ s'identifie à un sous-groupe fermé de $\bG_a^d$ pour un certain entier $d>0$ si, et seulement si, son décalage $V_G$ est nul. 
\end{theo}
\medskip
\begin{defi}
Soit $G$ un groupe algébrique sur $k$. On désigne par
\[
M(G):=\Hom_{\text{grp}/k}(G,\bG_a)
\]
le $k\{\tau\}$-module à gauche obtenu en pré-composant par les éléments de $\End_{\text{grp}/k}(\bG_a)$.
\end{defi}
\medskip
Rappelons que $\End_{\text{grp}/k}(\bG_a)$ est isomorphe à l'anneau non-commutatif $k\{\tau\}$, munissant alors $M(G)$ d'une structure de $k\{\tau\}$-module à gauche. La construction $G\mapsto M(G)$ se promeut en une équivalence de catégories (e.g.  cor. 6.7 {\it loc. cit.}):
\medskip
\begin{theo}\label{thm:equiv-cat}
L'assignation $G\mapsto M(G)$ définit une équivalence entre la catégorie des groupes algébriques commutatifs affines sur $k$ annulés par décalage, et la catégorie des $k\{\tau\}$-modules à gauche de type fini.  Un pseudo-inverse est donné par
\[
M\longmapsto U(M):\left(R\mapsto \Hom_{k\{\tau\}}(M,R)\right)
\]
où les homomorphismes sont pris dans la catégorie des $k\{\tau\}$-modules à gauches, et où la $k$-algèbre $R$ est vu comme un $k\{\tau\}$-module à gauche via $\tau\cdot r=r^p$.
\end{theo}
\medskip
\begin{rema}
Notons que cette équivalence fait correspondre aux modules à gauche libres sur $k\{\tau\}$ les puissances de $\bG_a$. Par abus, nous qualifierons ces groupes d'\emph{unipotents libres}.
\end{rema}
\medskip
Une première conséquence de cette classification et le résultat suivant :
\medskip
\begin{coro}\label{cor:grpalg-lisse-expp=unipo}
Si $G$ est un groupe algébrique connexe lisse d'exposant $p$, alors c'est une forme de $\bG_a^d$ pour un certain $d>0$, qui se déploie sur $k^{\operatorname{perf}}$.
\end{coro}
\medskip
\begin{proof}
Il vient du théorème de Barsotti-Chevalley que $G_{k^{\operatorname{perf}}}$ est commutatif affine: en effet, dans la suite exacte \eqref{eq:barsotti-chevalley} associée à $G_{k^{\operatorname{perf}}}$, la variété abélienne $E$ est d'exposant $p$ et est donc triviale. Ainsi, $G_{k^{\operatorname{perf}}}=H$ est commutatif affine.

Comme $G$ est d'exposant $p$, il en est de même de $H$ et on a l'égalité $F_H V_H=p=0$. Comme $H$ est lisse, $F_H$ est bijectif et cela entraîne $V_H=0$. $H$ est alors annulé par décalage et on conclut par le théorème \ref{thm:caract-unipotent} que $H$ est un fermé de $\bG_{a,k^{\operatorname{perf}}}^d$.  Il vient que $G$ est unipotent \cite[Cor. 15.9]{milne}.

Le $k^{\operatorname{perf}}\{\tau\}$-module à gauche $\Hom_{\text{grp}/k^{\operatorname{perf}}}(G_{k^{\operatorname{perf}}},\bG_{a,k^{\operatorname{perf}}})$ est de type fini et donc se décompose en la somme directe d'un module libre et d'un module de torsion (par le théorème de structure des $k^{\operatorname{perf}}\{\tau\}$-modules e.g. \cite[Prop. 1.4.4]{anderson}). Puisque $G$ est connexe, il en est de même pour $G_{k^{\operatorname{perf}}}$ et cela entraîne l'annulation de la partie de torsion.  D'après le théorème \ref{thm:equiv-cat}, on obtient que $G_{k^{\operatorname{perf}}}$ est unipotent libre.
\end{proof}

Une seconde conséquence est le résultat suivant démontré initialement par Russell (nous renvoyons à \cite[thm. 2.1 \& 3.1]{russell} pour la preuve de la proposition ci-dessous, conséquence du théorème \ref{thm:equiv-cat}).
\medskip
\begin{prop}\label{prop:russel}
Si $G$ est une forme de $\bG_a$, il existe deux entiers $n$ et $m$ ainsi que des éléments $a_0$, ..., $a_m$ de $k$, $a_m\neq 0$, tels que $\bG$ soit isomorphe au sous-groupe de $\bG_a^2=\Spec k[x,y]$ donné par l'équation $y^{p^n}=a_0x+a_1 x^p+...+a_m x^{p^m}$. Si $\bG \not\cong \bG_a$, alors $\End_{\text{grp}/k}(G)$ est un corps fini.  
\end{prop}

\subsection{Groupes algébriques affines}
On considère deux classes de groupes algébriques affines: ceux \emph{de type multiplicatif} et ceux \emph{unipotents}.  Sur un corps parfait, ces deux classes suffisent à décrire les groupes algébriques affines (voir théorème \ref{thm:grp-uni}). \\

On rappelle qu'un groupe algébrique  $M$ sur $L$ est \emph{de type multiplicatif} si et seulement si, pour une clôture séparable $L^s$ de $L$, il existe un groupe abélien $\Gamma$ de type fini tel que $M_{L^s}$ représente $R\mapsto \Hom_{\bZ} (\Gamma,R^\times)$. Les groupes algébriques de type multiplicatif sur $L$ sont uniquement déterminés par leur groupe de caractères (e.g. \cite[\S 14.f]{milne}):
\medskip
\begin{theo}\label{thm:multiplicatif}
Le foncteur
\[
M \longmapsto \Gamma:=\Hom_{\operatorname{grp}/L^s}(M_{L^s},\bG_{m,L^s})
\]
définit une équivalence de catégories entre celle des groupes algébriques de type multiplicatif sur $L$ et les $\bZ$-modules de type fini munis d'une action continue du groupe profini $G_L=\Gal(L^s|L)$. Un pseudo-inverse est donné par 
\[
\Gamma \longmapsto M(\Gamma):=\Spec (L^s[\Gamma]^{G_L})
\]
où $G_L$ agit diagonalement sur $L^s[\Gamma]$.
\end{theo}
\medskip

Nous nous inspirons de cette équivalence pour décrire le module de Tate d'un groupe algébrique de type multiplicatif $M$: 
\medskip
\begin{prop}\label{prop:multi}
Soit $p$ un nombre premier, $M$ un groupe algébrique sur $L$ de type  multiplicatif et $\Gamma$ son groupe des caractères. En tant que représentations de $G_L$,
\[
\operatorname{T}_p M \cong \Hom_{\bZ}\left(\varinjlim \Gamma/p^n \Gamma, (L^s)^{\times}\right)
\]
où $\sigma\in G_L$ agit à droite comme $f\mapsto \sigma\circ f\circ \sigma^{-1}$. En particulier,
\begin{enumerate}[label=$(\alph*)$]
    \item si $\Gamma$ a de la $p$-torsion, il en est de même pour $\operatorname{T}_p M$;
    \item le rang de $\operatorname{T}_p M$ sur $\bZ_p$ est celui de $\Gamma$ sur $\bZ$.
\end{enumerate}

\end{prop}
\begin{proof}
D'après le théorème \ref{thm:multiplicatif}, on a un isomorphisme de groupes
\begin{equation}\label{eq:idmult}
\theta:\Hom_{L}\!\left(L^s[\Gamma]^{G_L},L^s\right)=M(L^s)\cong M_{L^s}(L^s)\cong\Hom_{\bZ}(\Gamma,(L^s)^\times)
\end{equation}%
donné par $\varphi \mapsto (\varphi\otimes\id_{L^s[\Gamma]})|_{\Gamma}$ où $\Gamma$ est vu comme un sous-groupe de $(L^s[\Gamma],\times)$. En notant $M[p^n](L^s)$ la $p^n$-torsion de $M(L^s)$ pour un entier positif $n$, on a
\begin{align*}
M[p^n](L^s)&\cong\{f\in \Hom_{\bZ} (\Gamma,(L^s)^\times)~|~\forall\gamma\in\Gamma:~f(\gamma)^{p^n}=0\} \\
&\cong\Hom_{\bZ}\left(\Gamma/p^n \Gamma, (L^s)^{\times}\right)
\end{align*}
et l'expression de $\operatorname{T}_p M$ s'en déduit par passage à la limite sur $n$.\\
Il reste à déterminer l'action de $G_L$ sur $\operatorname{T}_p M$. Nous cherchons  l'unique action de $G_L$ sur $\Hom_{\bZ}(\Gamma,(L^s)^\times)$ qui rend $\theta$ équivariant pour l'action de $G_L$ (l'action sur $M[p^n](L^s)$ et $\operatorname{T}_p M$ s'en déduira par restriction et passage à la limite). Par définition, un automorphisme $\sigma\in G_L$ agit sur $(\varphi:L^s[\Gamma]^{G_L}\to L^s)\in M(L^s)$ par $\sigma\circ \varphi$. D'un autre côté, la flèche
\[
\sigma\circ(\varphi\otimes\id_{L^s[\Gamma]})\circ \sigma^{-1},
\]
où $G_L$ agit diagonalement sur  $L^s[\Gamma]$, est une fonction $L^s$-linéaires qui coïncide avec $\sigma\circ\varphi$ sur $(L^s[\Gamma])^{G_L}$. Une telle fonction étant unique, on a
\[
(\sigma\circ\varphi)\otimes\id_{L^s[\Gamma]}=\sigma\circ(\varphi\otimes\id_{L^s[\Gamma]})\circ \sigma^{-1},
\]
puis
\[
\theta(\sigma\circ\varphi)=\sigma\circ\theta(\varphi)\circ \sigma^{-1}
\]
par restriction à $\Gamma$.
\end{proof}

En utilisant \cite[thm. 17.17+cor.15.17-18]{milne}, on en déduit le théorème annoncé:
\medskip
\begin{theo}\label{thm:grp-uni}
Tout groupe algébrique affine $H$ sur un corps parfait se décompose comme $H\cong U\times M$ où $U$ est unipotent et $M$ est de type multiplicatif. Si $H$ est de plus un schéma en $A$-modules, il en est de même canoniquement pour $U$ et $M$.
\end{theo}

\section{Démonstrations}\label{sec:demo}
On démontre ici les théorèmes énoncés en introduction. Les preuves étant très sensibles à la caractéristique de $A$, nous commencerons par le cas de caractéristique nulle en sous-section \ref{subsec:nulle} puis le cas plus sophistiqué de la caractéristique non nulle en sous-section \ref{sec:modules-elementaire-carnonnulle}.

\subsection{Le cas de la caractéristique nulle}\label{subsec:nulle}
Soit $A$ un anneau de Dedekind finiment engendré de caractéristique nulle. Soit $K$ son corps des fractions; c'est un corps de nombres d'anneau d'entiers $A$. Soit $L$ un corps parfait. Nous prouvons ici les théorèmes \ref{thm:coefficient}.\ref{thm:coefficient-i} et \ref{thm:classification}.\ref{thm:classification-i}. Précisément: 
\medskip
\begin{theo}\label{thm:cas-car-zero}
Soit $\bG$ un schéma en $A$-modules sur $L$ connexe lisse de dimension $1$ et de rang $r\geq 1$. Alors,
\begin{enumerate}
\item\label{item:A=Z} soit $A=\bZ$, auquel cas soit $r=1$ et $\bG$ est une forme de $\bG_m$, soit $r=2$ et $E$ est une courbe elliptique;
\item\label{item:A=Ok} soit $A=\cO_K$ où $K$ est une extension quadratique imaginaire de $\bQ$, auquel cas $r=1$ et $E$ est une courbe elliptique ayant multiplication complexe par $\cO_K$. 
\end{enumerate}
\end{theo}
\medskip
\begin{proof}
En appliquant la proposition \ref{prop:barsotti-chevalley-A}, on obtient une suite exacte de schémas en $A$-modules sur $L$:
\[
1\longrightarrow H \longrightarrow \bG \stackrel{f}{\longrightarrow} E \longrightarrow 1.
\]
Les fibres de $f$ sont les classes de $H$ dans $\bG$, qui ont toutes la même dimension, et donc
\begin{equation}\label{eq:dimension}
\dim H= \dim f^{-1}(\{1_E\})=\dim \bG-\dim E
\end{equation}
(voir par exemple \cite[cor. 14.119]{gortz-wedhorn}). Comme $\dim \bG=1$, nous avons $\dim H\leq 1$. Pour déterminer $H$, on commence par remarquer que:
\medskip
\begin{prop}\label{prop:H}
$H$ est soit trivial, soit une forme de $\bG_m$.
\end{prop}
\medskip
La proposition ci-dessus se déduit du lemme suivant:
\medskip
\begin{lemm}\label{lem:TpG-sanstorsion}
Soit $p$ un nombre premier. Alors $\operatorname{T}_p \bG$ est sans torsion.
\end{lemm}
\medskip
\begin{proof}
Si $\operatorname{T}_p \bG$ possède de la torsion, alors $\bG(L^s)$ possède de la $p$-torsion. Il existe donc $\ell$ un idéal de $A$ au-dessus de $p A$ tel que $\bG(L^s)$ ait de la $\ell$-torsion. Mais c'est impossible car $\operatorname{T}_\ell \bG$ est libre sur $A_\ell$.
\end{proof}

\begin{proof}[Démonstration de la proposition \ref{prop:H}]
On suppose sans perte $H$ non trivial. Puisque $H$ est un groupe algébrique affine sur un corps parfait, il se décompose comme $U\times M$ où $U$ est unipotent et $M$ est de type multiplicatif (théorème \ref{thm:grp-uni}).

On prétend que $U$ est trivial. En effet, si $\dim U= 1$, alors $\dim H=1$ puis $\dim M=\dim E=0$ (par \eqref{eq:dimension}) et donc $\bG=H$. Pour un nombre premier $p$, on a $\operatorname{T}_p \bG=\operatorname{T}_p U\oplus \operatorname{T}_p M=\operatorname{T}_p M$. Mais puisque $\dim M=0$, $\operatorname{T}_p \bG$ serait un $\bZ_p$-module de torsion ce qui est absurde par le Lemme \ref{lem:TpG-sanstorsion}. Ainsi, $\dim U=0$ et, étant unipotent sur un corps de caractéristique $0$, $U$ est trivial. 

On sait donc que $H$ est de type multiplicatif et on note $\Gamma_H$ son groupe des caractères. On a de plus une inclusion $\operatorname{T}_p H\hookrightarrow \operatorname{T}_p \bG$ et, comme $\operatorname{T}_p \bG$ est sans torsion, il en est de même de $\operatorname{T}_p H$. Comme $\operatorname{rang}_{\bZ} \Gamma_H=\dim H=1$, la proposition \ref{prop:multi} donne $\Gamma_H\cong \bZ$. $H$ est donc une forme de $\bG_m$. 
\end{proof}

Il reste deux cas à considérer: comme $\dim \bG=1$, alors $\dim E\leq 1$ et donc soit $E$ est trivial, soit $E$ est une courbe elliptique.  \\
\emph{Si $E$ est trivial:} alors $\bG=H$. C'est donc une forme de $\bG_m$ par la proposition \ref{prop:H}. En particulier, $\End_L(\bG)\cong \bZ$ et puisque $A\hookrightarrow \End_L(\bG)$ par le lemme \ref{lem:phi-injective}, cela impose $A=\bZ$. \\
\emph{Si $E$ est une courbe elliptique:} alors $\dim H=0$ et donc $H$ est trivial par la proposition \ref{prop:H}. $\bG$ est donc une courbe elliptique. Puisque $L$ est de caractéristique nulle, $\End_{L}(E)$ est soit isomorphe à $\bZ$, soit à un ordre d'un corps quadratique imaginaire $F$. Par le lemme \ref{lem:phi-injective}, on a $A\hookrightarrow \End_L(E)$ et, puisque $A$ est Dedekind, cela impose $A=\bZ$ ou $A=\End_L(E)\cong \cO_F$ auquel cas $K\cong F$.
\end{proof}

\begin{rema}
Comme mentionné en remarque \ref{rem:presque-vrai}, la réciproque du théorème \ref{thm:cas-car-zero} est \emph{preque vraie} à ceci près qu'une courbe elliptique à multiplication complexe est un module élémentaire de type $(1)$ si et seulement si son morphisme caractéristique coïncide à l'inclusion. Cette \emph{presque réciproque} se déduit des propositions \ref{prop:reciproque1} et \ref{prop:courbe-elliptique-cm}. 
\end{rema}

\subsection{Le cas de la caractéristique non nulle}\label{sec:modules-elementaire-carnonnulle}
\`A présent, et jusqu'à la fin de cet article, on suppose que $A$ est un anneau de Dedekind finiment engendré de caractéristique $p>0$. Nous désignerons par $K$ son corps des fractions. Soit $L$ un corps qui est également une $A$-algèbre via un morphisme $\delta:A\to L$. Notons $\bF_q$ la clôture algébrique de $\bF_p$ dans $A$: c'est une extension finie de $\bF_p$ par hypothèse, et l'on note $q$ son nombre d'éléments.

\subsubsection*{Fin de la preuve du théorème \ref{thm:coefficient}}
On démontre ici la fin du théorème \ref{thm:coefficient}, à savoir:
\medskip
\begin{theo}\label{thm:iithm1}
Soit $\bG$ un schéma en $A$-modules connexe et lisse sur $L$ de rang $r\geq 1$ et de dimension $1$. Alors il existe une courbe projective lisse $(C,\cO_C)$ sur $\bF_p$ ainsi qu'un point fermé $\infty$ de $C$ tels que $A$ est isomorphe à $\cO_C(C\setminus \{\infty\})$.
\end{theo}
\medskip
Il s'agit dans un premier temps d'énoncer une caractérisation équivalente au fait d'être l'anneau des fonctions régulières sur une courbe projective lisse privé d'un point. Nous dirons d'un élément $a\in A$ qu'il est {\it constant} s'il est algébrique sur $\bF_p$. On considère la propriété suivante:
\begin{enumerate}[label=$(P_A)$]
    \item\label{prPA} {\it Pour tout $a\in A$ non constant, $A$ est un $\bF_p[a]$-module de type fini.}
\end{enumerate}
Alors nous avons:
\medskip
\begin{lemm}\label{prop:caracterisation-courbe-moins-un-point}
L'anneau $A$ vérifie \ref{prPA} si, et seulement si, on peut trouver une courbe projective lisse $(C,\cO_C)$ sur $\bF_p$ ainsi qu'un point fermé $\infty$ de $C$ pour lesquels 
\[
A=\operatorname{H}^0(C\setminus\{\infty\},\cO_C).
\]
\end{lemm}
\begin{proof}
L'une des deux directions est bien connue: soit $(C,\cO_C)$ une courbe projective lisse sur $\bF_p$ ainsi qu'un point fermé $\infty$ de $C$. Soit $B=\operatorname{H}^0(C\setminus \{\infty\},\cO_C)$. Alors la propriété $(P_B)$ est vérifiée. En effet, soit $b\in B$ un élément non constant, i.e. la flèche $\bF_p[t]\to B$ qui envoie $t\mapsto b$ est injective. On a alors une inclusion de corps $\bF_p(t)\subset K=\Frac(B)$ qui, par l'équivalence de catégorie entre corps de fonctions sur $\bF_p$ et courbes projectives lisses sur $\bF_p$ \href{https://stacks.math.columbia.edu/tag/0BY1}{[0BY1]}, produit un morphisme $C\to \bP^1$ de schémas sur $\bF_p$. Puisque $\bP^1$ est séparé, ce dernier est automatiquement propre \href{https://stacks.math.columbia.edu/tag/01W6}{[01W6]}, et on notera que l'unique point de $C$ qui n'est pas envoyé sur $\Spec \bF_p[t]$ est le point $\infty$. Puisque que la propriété d'être {\it propre} est stable par changement de base \href{https://stacks.math.columbia.edu/tag/01W4}{[01W4]}, le morphisme
\[
\Spec B=C\times_{\bP^1} \Spec \bF_p[t] \longrightarrow \bP^1\times_{\bP^1}\Spec \bF_p[t]=\Spec \bF_p[t] 
\]
est lui même propre. \'Etant également affine \href{https://stacks.math.columbia.edu/tag/01SH}{[01SH]}, il est fini \href{https://stacks.math.columbia.edu/tag/01WN}{[01WN]}. 

Montrons maintenant  la réciproque. Puisque $A$ n'est pas un corps, il existe $a\in A$ non constant. L'anneau $A\otimes_{\bF_p[a]}\bF_p(a)$ est intègre (par platitude) et, par \ref{prPA}, fini sur le corps $\bF_p(a)$. Il vient que $A\otimes_{\bF_p[a]}\bF_p(a)=K$ où $K$ est le corps des fractions de $A$. $K$ est alors une extension finie de $\bF_p(a)=\bF_p(\bP^1)$, d'où l'existence d'une courbe $(C,\cO_C)$ projective lisse sur $\bF_p$ telle que $\Frac A=\bF_p(C)$. Soit $s$ un point fermé de $C$ dans le complémentaire de $S=\operatorname{Spm}A$ (le complémentaire est non vide car sinon $A$ serait le corps des constantes). L'anneau $B=\operatorname{H}^0(C\setminus \{s\},\cO_C)$ est un sous-anneau de $A$ qui est Dedekind et, pour $b\in B$ non constant, l'inclusion $\bF_p[b]\subset A$ est finie par \ref{prPA}. Ainsi, $B\subset A$ est finie, et puisque $B$ est intégralement clos dans $K$, il vient $A=B$.
\end{proof}

Pour démontrer le théorème \ref{thm:iithm1}, il suffit donc d'établir \ref{prPA}. Soit $\bG$ comme dans le théorème \ref{thm:iithm1}.
\medskip
\begin{prop}\label{prop:G-est-Ga}
En tant que groupes algébriques sur $L$, $\bG$ est isomorphe à $\bG_a$.
\end{prop}
\medskip
\begin{proof}
Par le corollaire \ref{cor:grpalg-lisse-expp=unipo}, $\bG$ est une forme de $\bG_a$. Par le lemme \ref{lem:phi-injective}, on a que $\End_L(\bG)$ est infini. En particulier, $\bG\cong \bG_a$ d'après la proposition  \ref{prop:russel}.
\end{proof}

Soit donc $\square:\bG\stackrel{\sim}{\to}\bG_a$ un isomorphisme de groupes algébriques sur $L$.  La composition
\begin{equation}
\begin{tikzcd}[column sep=4em]
\End_{\text{grp}/L}(\bG) \arrow[r,"f\mapsto \square f\square^{-1}"] & \End_{\text{grp}/L}(\bG_a)\stackrel{\sim}{\longrightarrow} L\{\tau\}
\end{tikzcd} \nonumber
\end{equation}
induit un morphisme (non canonique) $A\to L\{\tau\}$. On notera 
\begin{equation}\label{eq:Phia}
\Phi_a(\tau)=\Phi^{\square}_a(\tau)\in L\{\tau\}
\end{equation}
le polynôme en $\tau$ associé à $a\in A$ de cette façon.
\medskip
\begin{lemm}\label{lem:inversible=constant}
Les éléments inversibles de $A$ sont les éléments constants non nuls.
\end{lemm}
\medskip
\begin{proof}
Comme $A$ est intègre, tout élément constant non nul est inversible. Supposons qu'il existe $a\in A$ inversible mais non constant. On prétend alors qu'il existe un polynôme $P(x)\in \bF_p[x]$ non constant tel que $P(a)\not\in A^{\times}$. En effet, si ce n'est pas le cas, alors $\bF_p(a)$ est un sous-corps de $A$. Mais puisque $A$ est de type fini sur $\bF_p$ , pour un idéal maximal $\fm\subset A$ on aurait $\bF_p\subset \bF_p(a)\subset A/\fm$, ce dernier étant une extension finie du premier. Mais c'est absurde, puisque $\bF_p(a)$ n'est pas fini sur $\bF_p$.

Pour un tel $P(x)$,  le groupe 
\begin{equation}\label{eq:torsion-P(a)}
\bG[P(a)](L^s)\cong \left\{ x\in L^s~ |~ \Phi_{P(a)}(\tau)(x)=0\right\}
\end{equation}
est fini et non trivial (pour l'isomorphisme ci-dessus on a utilisé la Proposition \ref{prop:G-est-Ga}). Mais comme $\Phi_a\in (L\{\tau\})^\times=L^\times$, on a $\Phi_{P(a)}=P(\Phi_a)\in L^{\times}$, ce qui contredit la finitude (ou la non trivialité) du groupe \eqref{eq:torsion-P(a)}.
\end{proof}

\begin{proof}[Démonstration du théorème \ref{thm:iithm1}]
Soit $a\in A$ non constant. On veut montrer que $\bF_p[a]\to A$ est fini. Par le lemme \ref{lem:inversible=constant}, $a$ n'est pas inversible, et donc $\bG[a](L^s)$ est fini et non trivial. Ainsi, $\deg_{\tau}\Phi_a>0$. 

On munit $L\{\tau\}$ d'une structure de $L\otimes_{\bF_p}A$-module où $L$ agit à gauche et $b\in A$ à droite par $\Phi_b(\tau)$. Pour un certain $\rho\in \Gal(\bF_q|\bF_p)$, cette structure de module se factorise par celle de $L\otimes_{\rho,\bF_q}A$-module. Puisque $\deg_{\tau}\Phi_a>0$, par division euclidienne à gauche dans $L\{\tau\}$ on obtient que c'est un $L[a]=L\otimes_{\rho,\bF_q}\bF_q[a]$-module de type fini. En particulier, on a des morphismes de $L$-algèbres:
\[
L[a]\longrightarrow L\otimes_{\rho,\bF_q} A\stackrel{h}{\longrightarrow} L\{\tau\}
\]
dont la composition est de type fini. Pour montrer que le premier morphisme est de type fini, il suffit, par Noetherianité, de montrer que le second est injectif. Mais c'est clair car si $\ker h\neq (0)$ alors $L\otimes_{\rho,\bF_q} A/(\ker h)$ serait de dimension fini sur $L$, ce qui contredit le fait que l'image de $h$ est de dimension infinie sur $L$ (car contient $\Phi_a(\tau)$ et ses puissances).

Nous avons alors montré que $L[a]\to L\otimes_{\rho,\bF_q} A$, et donc $L[a]\to L\otimes_{\bF_p} A$, sont des morphismes finis. La finitude de $A$ sur $\bF_p[a]$ s'en déduit par fidèle platitude. 

\end{proof}

\subsubsection*{Preuve du théorème \ref{thm:classification} pour les modules élémentaires de type $(1)$}
On termine ici la preuve de \ref{thm:classification}.\ref{thm:classification-ii}. Soit $\bG$ comme dans le théorème \ref{thm:iithm1}. On fixe une courbe projective lisse $(C,\cO_C)$ sur $\bF_p$ ainsi qu'un point fermé $\infty$ de $C$ pour lesquels
\[
A=\cO_C(C\setminus\{\infty\}).
\]
Soit $\bF_q\subset A$ le corps des éléments constants et soit $q$ son nombre d'éléments. Soit $e=\log_p q$. \\

On liste dans la proposition suivante quelques propriétés des polynômes $\Phi_a(\tau):=\Phi^\square_a(\tau)$ définis en \eqref{eq:Phia}. Notons $\delta_\bG$ le morphisme caractéristique de $\bG$ (définition \ref{def:morphisme-cara}).
\medskip
\begin{prop}\label{prop:propriete-Phia}
\begin{enumerate}[label=$(\roman*)$]
\item\label{item:1} Pour tout $a\in A$, $\Phi_a(\tau)\in L\{\tau^e\}$.
\item\label{item:2bis} Le terme constant de $\Phi_a(\tau)$ est $\delta_\bG(a)$.
\item\label{item:2} Il existe $t\in \{0,...,e-1\}$, tel que  pour tout $c\in \bF_q$,  $\Phi_c(\tau)=c^{p^t}$.
\item\label{item:3} Pour tout $a\in A$ non nul, $\deg_\tau \Phi_a(\tau)=re\deg(a)$.
\end{enumerate}
\end{prop}
\medskip
\begin{proof}
Si l'on restreint l'action de $A$ à celle de $\bF_q$, cela induit sur $\bG$ une structure de schéma en $\bF_q$-espaces vectoriels connexe lisse pour laquelle les endomorphismes de $\varphi(A)$ sont $\bF_q$-linéaires. L'assertion \ref{item:1} est alors  conséquence directe de la relation $\Phi_a(\tau)\Phi_c(\tau)=\Phi_c(\tau)\Phi_a(\tau)$ pour tout $a\in A$ et $c\in \bF_q$. 

Le point \ref{item:2bis} découle de la commutativité du diagramme d'anneaux:
\begin{equation}
\begin{tikzcd}
\Phi^\square:~A \arrow[r,hook] \arrow[dr,"\delta_\bG"'] & \End_{\operatorname{grp}/L}(\bG) \arrow[r,"\square"] \arrow[d,"\Lie"] & \End_{\operatorname{grp}/L}(\bG_a) \arrow[r,"\sim"]\arrow[d,"\Lie"] & L\{\tau\}\arrow[d,"\partial"] \\
 & \End_{L\operatorname{-ev}/L}(\Lie_\bG) \arrow[r,"\Lie_\square"] & \End_{L\operatorname{-ev}/L}(\bG_a) \arrow[r,"\sim"] & L
\end{tikzcd}
\end{equation}
obtenu par fonctorialité de $\bG\mapsto \Lie_\bG$, où la flèche verticale $\partial:L\{\tau\}\to L$ envoie un polynôme $P(\tau)=a+b\tau+...$ vers son terme constant $a$.

Pour \ref{item:2}, il suffit de remarquer que $\delta_\bG(\bF_q^\times)\subset (L\{\tau\})^\times=L^\times$. Ainsi, par \ref{item:2bis}, $\Phi_a(\tau)=\delta_\bG(a)$ pour tout $a\in \bF_q$. Or, $\delta$ induit par restriction un morphisme de corps $\bF_q\to \bF_q\subset L$ et coïncide alors avec une puissance $t$ du $p$-Frobenius par Théorie de Galois des corps finis. 

Montrons \ref{item:3}. Soit $a\in A$ non nul de degré $d>0$. Par \ref{item:2}, on sait que pour tout $a\in \bF_q$, on a $\deg_\tau \Phi_a(\tau)=\deg_\tau \Phi_{a+c}(\tau)$ et $\deg(a+c)=\deg(a)$ (car $[K:\bF_q(a)]=[K:\bF_q(a+c)]$). Quitte à remplacer $a$ par $a+c$, on peut donc supposer que le coefficient constant de $\Phi_a(\tau)$ est non nul. Les solutions dans le corps $L^s$ de 
\begin{equation}\label{eq:phia=0}
\Phi_a(x)=a_0x+a_1x^q+...+a_sx^{q^s}=0 \quad (x\in L^s)
\end{equation}
coïncident avec les éléments de $\bG[a](L^s)$. D'une part, l'hypothèse $a_0\neq 0$ montre que l'équation \eqref{eq:phia=0} admet $q^s$ solutions. D'autre part, l'ensemble $\bG[a](L^s)$ a $q^{rd}$ éléments, et par conséquent $s=rd$. En effet, $\bG[a](L^s)$ est isomorphe à $\bG[\ell_1^{c_1}](L^s)\times \cdots \times \bG[\ell_s^{c_s}](L^s)$ et, par définition des modules élémentaires, a $q^{rc_1 \deg(\ell_1)}\cdots q^{rc_s\deg \ell_s}$ éléments. Cela découle donc de l'additivité des degrés. 
\end{proof}

La proposition suivante correspond l'énoncé du théorème \ref{thm:classification}, partie \ref{thm:classification-ii} pour les modules élémentaires de type $(1)$.
\medskip
\begin{prop}\label{prop:Gdrinfeld}
$\bG$ est un module de Drinfeld sur $L$ de coefficient $A$, de rang $r$, et de caractéristique $\delta_\bG$.
\end{prop}
\medskip
\begin{proof}
Montrons que le module $M(\bG)=\Hom_{\operatorname{grp}/L}(\bG,\bG_a)$, muni de sa structure canonique de $A\otimes_{\bF_p}L$-module, et localement libre de rang $r$. Après un choix de coordonnées $\square$, $M(\bG)$ devient isomorphe à $L\{\tau\}$ où $l\in L$ agit sur $p(\tau)\in L\{\tau\}$ par $l\cdot p(\tau)$ et $a\in A$ par $p(\tau)\cdot \Phi_a(\tau)$.

En suivant la décomposition de $A\otimes_{\bF_p}L$ par idempotents,
\begin{equation}\label{eq:idempotent}
A\otimes_{\bF_p} L\cong \bigoplus_{s\in \bZ/e\bZ}A\otimes_{\bF_q,c\mapsto c^{p^s}}L,
\end{equation}
le module $M(\bG)\cong_{\square} L\{\tau\}$ se décompose uniquement comme somme directe de sous-modules $M(\bG)_s$, $s\in \bZ/e\bZ$, où  $M(\bG)_s$ est le sous-module des éléments $m\in M(\bG)$ sur lesquels $c\in \bF_q\subset A$ agit comme $c^{p^s}\in \bF_q\subset L$. Comme $c\in \bF_q\subset A$ agit par $p(\tau)\mapsto p(\tau)c^{p^t}$ (proposition \ref{prop:propriete-Phia}.\ref{item:2}), on trouve
\[
\text{Pour~}s\in \{t,...,t+e-1\}, \quad M(\bG)_{\bar{s}}\cong_{\square} L\{\tau^e\}\tau^{s-t}.
\]
Soit $a\in A$ non constant. Par division euclidienne par $\Phi_a(\tau)$ dans $L\{\tau^e\}\tau^{s-t}$, on en déduit que $M(\bG)_{\bar{s}}$ est un $\bF_q[a]\otimes_{\bF_q,x\mapsto x^{p^s}}L$-module libre de rang $r\deg(a)$, de base $(\tau^{s-t},\tau^{e+s-t},...,\tau^{(r\deg(a)-1)e+s-t})$ (proposition \ref{prop:propriete-Phia}.\ref{item:3}). Le module $M(\bG)_{\bar{s}}$ est donc de type fini et sans torsion sur l'anneau de Dedekind $A\otimes_{\bF_q,c\mapsto c^{p^s}}L$; il est en particulier localement libre de rang constant. Comme $A\otimes_{\bF_q,c\mapsto c^{p^s}}L$ est un $\bF_q[a]\otimes_{\bF_q,c\mapsto c^{p^s}}L$-module libre de rang $\deg(a)$, on en déduit que $M(\bG)_{\bar{s}}$ est de rang $r$ sur $A\otimes_{\bF_q,c\mapsto c^{p^s}}L$.

Le module $M(\bG)$ est donc localement libre de rang $r$ sur $A\otimes_{\bF_p}L$. On en déduit que $\bG$ est un module de Drinfeld. Qu'il soit de caractéristique $\delta_\bG$ découle de la définition. 
\end{proof}

\subsubsection{Preuve du théorème \ref{thm:classification} pour les modules élémentaires de type $(2)$}
On termine ici la preuve du théorème \ref{thm:classification} pour les modules élémentaires de type $(2)$. On suppose donc que $L$ une extension finie de $K$. Notons $\cO_L$ la clôture intégrale de $A$ dans $L$. Pour $\fP$ un idéal maximal de $\cO_L$, on note $\cO_{(\fP)}\subset L$ le sous-anneau de $L$ formé des éléments $\fP$-entiers:
\[
\cO_{(\fP)}:=\{x\in L~|~v_\fP(x)\geq 0\}
\]
où $v_\fP$ désigne une valuation sur $L$ associée à $\fP$. La notation $\cO_\fP$ reste réservée à la complétion de $\cO_L$ (resp. $\cO_{(\fP)}$) en l'idéal $\fP$, de corps résiduel $\bF_\fP$. On notera également $L_\fP$ le corps des fractions de $\cO_\fP$. Enfin, soit $L^s$ et $L_\fP^s$ des clôtures séparables de $L$ et $L_\fP$ respectivement, et soit $\bar{\cO}_\fP$ la clôture intégrale de $\cO_\fP$ dans $L_\fP^s$. \\

Supposons cette fois-ci que $\bG$ est un module élémentaire de type $(2)$ de dimension $1$, de coefficient $A$, de base $L$ et de rang $r$. Rappelons que cela impose les propriétés suivantes: il existe un ensemble fini $S$ d'idéaux maximaux de $\cO_L$ pour lequel, pour tout $\fP$ un idéal maximal de $\cO_L$ en dehors de $S$ et $\ell$ un idéal maximal de $A$ différent de $\fp:=\fP\cap A$,
\begin{enumerate}[label=$(\alph*)$]
    \item La représentation $\operatorname{T}_\ell \bG$ est non ramifiée en $\fP$; i.e. le groupe d'inertie $I_\fP\subset G_L$ en $\fP$ agit trivialement, et, 
    \item Le déterminant $s(\fP)\in A_\ell$ de l'action de $\Frob_\fP\in G_L/I_\fP$ sur $\operatorname{T}_\ell \bG$ appartient à $A$ et est indépendant de $\ell$.
\end{enumerate}

Soit $\delta=\delta_\bG:A\to L$ le morphisme caractéristique de $\bG$ (définition \ref{def:morphisme-cara}). Pour finir la preuve du théorème \ref{thm:classification}, il nous faut montrer que $\delta$ coïncide avec une puissance du Frobenius. Pour la suite, on ne perd rien à supposer que $r=1$: en effet, on sait par la proposition \ref{prop:Gdrinfeld} que $\bG$ est un module de Drinfeld et, quitte à remplacer $\bG$ par $\det \bG$ qui, par le théorème \ref{thm:det}, est un module de Drinfeld de rang $1$ de même caractéristique, on peut supposer $r=1$.

\subsubsection{Cœur de l'argument}
La preuve se déroule en deux étapes: dans un premier temps, on montre (proposition \ref{prop:psi-res-frob}) que $\delta:A\to L$ est $\fP$-entier pour presque toute place finie $\fP$ de $L$ puis que, pour tout $a\in A$ et tout tel $\fP$, il existe un entier $k=k_{a,\fP}$ pour lequel 
\[
\delta(a)\equiv a^{p^k} \pmod{\fP}.
\]
Dans un second temps, on montre un résultat général qui prouve qu'une telle condition sur $\delta$ suffit à en déduire que c'est une puissance du morphisme de Frobenius (théorème \ref{theo:frob-res}). 
\medskip
\begin{prop}\label{prop:psi-res-frob}
Pour presque tout idéal premier $\fP$ de $\cO_L$ et tout élément $a\in A$, il existe un entier positif $k$ (dépendant de $\fP$ et de $a$) tel que $\delta(a)\equiv a^{p^k}$ modulo $\fP$.
\end{prop}
\medskip

La preuve de cette proposition résultera du lemme \ref{lem:en-trois-points} qui suit. \\

Soit $S_T$, "T" pour {\it total}, l'ensemble des idéaux maximaux $\fP$ de $\cO_L$ pour lesquels il existe $a\in A$ tel que l'un des coefficients de $\Phi_a(\tau)$ ne soit pas $\fP$-entier. Comme $A$ est finiment engendré en tant qu'anneau, $S_T$ est fini. En particulier, pour $\fP$ en dehors de $S_T$, l'image de $\delta$ est contenue dans $\cO_{(\fP)}$. \\

Pour $a\in A$ non nul, notons $\varepsilon(a)\in L^\times$ le coefficient dominant de $\Phi_a(\tau)$. On conviendra que $\varepsilon(0)=0$. Les relations de commutation des $\Phi_a(\tau)$ ainsi que l'assertion \ref{item:3} de la proposition précédente entraînent que:
\[
\forall a,b \in A, \quad \varepsilon(a)\varepsilon(b)^{q^{r\deg a}}\!=\varepsilon(b)\varepsilon(a)^{q^{r\deg b}}.
\]
En particulier, le support premier de $\varepsilon(a)$ (i.e. l'ensemble les idéaux maximaux $\fP$ de $\cO_L$ tels que $v_\fP(\varepsilon(a))\neq 0$) est indépendant de $a\neq 0$. Notons le  $S_D$, "D" pour {\it dominant}. \\

\'Etant donné un idéal maximal $\fP$ de $\cO_L$ en dehors de $S_T$, $\fa$ un idéal de $A$, on introduit le foncteur :
\[
\mathbf{Z}_\fP[\fa]:\mathbf{Alg}_{\cO_{(\fP)}}\longrightarrow \mathbf{Ens}
\]
qui à une $\cO_{(\fP)}$-algèbre $R$ assigne l'ensemble fini de $R$:
\[
\mathbf{Z}_{\fP}[\fa](R):=\{x\in R~|~\forall a\in \fa:~\Phi_a(\tau)(x)=\delta(a)x+...+\varepsilon(a)x^{q^{\deg a}}=0\}.
\]
Restreint à la catégorie des $L$-algèbres, $\mathbf{Z}_{\fP}[\fa](R)$ est équivalent au foncteur $R\mapsto \bG[\fa](R)$. 

\medskip
\begin{lemm}\label{lem:en-trois-points}
Soit $\fP$ un idéal maximal de $\cO_L$ en dehors de l'ensemble fini $S\cup S_T\cup S_D$ et soit $\ell$ un idéal maximal de $A$ tel que $\deg \ell>\deg \fP$. Alors,
\begin{enumerate}
\item\label{lem:psi(l)-premier-a-p} Si $\lambda\in A$ et $n\geq 1$ sont tels que $\ell^n=(\lambda)$, alors $v_\fP(\delta(\lambda))=0$.
\item\label{prop:torsion-reduction} Pour une infinité d'entiers $n\geq 1$, les flèches obtenues par naturalité:
\[
\mathbf{Z}_{\fP}[\ell^n](L^s)\longrightarrow \mathbf{Z}_{\fP}[\ell^{n}](L^s_{\fP})\longleftarrow \mathbf{Z}_{\fP}[\ell^{n}](\bar{\cO}_{\fP}) \longrightarrow \mathbf{Z}_{\fP}[\ell^n](\bar{\bF}_\fP)
\]
sont des bijections.
\item\label{lem:vp-sp>0} On a $\deg s(\fP)=\deg \fP$ et $v_\fP(\delta(s(\fP)))>0$.
\end{enumerate}
\end{lemm}
\medskip
\begin{proof}
Pour le point \ref{lem:psi(l)-premier-a-p}, le noyau de la composition $\delta_\fP:A\stackrel{\delta}{\to} \cO_\fP\to \bF_\fP$ définit un idéal maximal de $A$ de degré $\leq \deg \fP$. Si $v_\fP(\delta(\lambda))>0$, alors $\lambda\in \ker \delta_\fP$, soit $(\ker \delta_\fP)| \ell^c$ pour un certain $c>0$, puis $\ell=\ker \delta_\fP$ par primalité. Cela contredit notre hypothèse sur le degré de $\ell$. \\
Démontrons le point \ref{prop:torsion-reduction} lorsque $n$ est un multiple de $h$, le cardinal de $\operatorname{Cl}(A)$. Soit $\lambda\in A$ un générateur de l'idéal principal $\ell^h$. En vertu du point \ref{lem:psi(l)-premier-a-p}, le polynôme $\Phi_\lambda(\tau)$ s'écrit
\[
\Phi_\lambda(\tau)=\delta(\lambda)+(\lambda)_1\tau^e+...+\varepsilon(\lambda)\tau^{re\deg \lambda} \in \cO_{(\fP)}\{\tau\},
\]
avec $\delta(\lambda)$ et $\varepsilon(\lambda)$ inversible dans $\cO_{(\fP)}$. En particulier, le polynôme 
\[
\Phi_\lambda(\tau)(x)=\delta(\lambda)x+(\lambda)_1x^q+...+\varepsilon(\lambda)x^{q^{\deg \lambda}}
\]
est séparable -- que donc $\mathbf{Z}_{\fP}[\ell^n](L^s)= \mathbf{Z}_{\fP}[\ell^{n}](L^s_{\fP})$ -- et ses racines sont $\fP$-entières dans toute extension séparablement close de $L$. On en déduit $\mathbf{Z}_{\fP}[\ell^{n}](L^s_{\fP})=\mathbf{Z}_{\fP}[\ell^{n}](\bar{\cO}_{\fP})$. Pour la dernière égalité, il suffit d'observer que $\Phi_\lambda(\tau)(x)$ est à racines simples à la fois dans $\bar{\cO}_\fP$ et $\bar{\bF}_\fP$, et que les facteurs irréductibles de  $\Phi_\lambda(\tau)(x)$ sur $\bar{\bF}_\fP$ proviennent de ceux sur $\bar{\cO}_\fP$. \\
Enfin démontrons le point \ref{lem:vp-sp>0}. Soit $\ell\subset A$ un idéal maximal de degré supérieur à $\deg \fP$. Puisque $\bG$ est de rang $1$ et $\fP$ en dehors de $S$, l'automorphisme $\Frob_\fP$ agit sur $\operatorname{T}_\ell \bG$ par multiplication par $s(\fP)$. D'après le point \ref{prop:torsion-reduction}, cela signifie que $\Phi_{s(\fP)}(\tau)$ agit sur $\mathbf{Z}_\fP[\ell^n](\bar{\bF}_{\fP})$ par $x\mapsto x^{q^{\deg \fP}}$. Par conséquent, pour tout $\gamma\in \mathbf{Z}_\fP[\ell^n](\bar{\bF}_\fP)\subset \bar{\bF}_\fP$, nous avons la congruence:
\[
\Phi_{s(\fP)}(\tau)(\gamma)\equiv \gamma^{q^{\deg \fP}}
\]
dans $\bar{\bF}_\fP$. En prenant $\ell$ de degré assez grand, on observe que cette identité est vérifiée pour assez de $\gamma\in \bar{\bF}_\fP$ pour être relevée en une identité polynomiale:
\[
\Phi_{s(\fP)}(\tau)(X)\equiv X^{q^{\deg \fP}}.
\]
Puisque $\fP$ est choisi en dehors de $S_D$, $\varepsilon(s(\fP))$ est une unité de $\cO_\fP$ et on trouve $\deg(s(\fP))=\deg \fP$. Il vient également $v_\fP(\delta(s(\fP)))>0$.
\end{proof}

\begin{proof}[Preuve de la Proposition \ref{prop:psi-res-frob}]
Soit $\fP$ en dehors de $S\cup S_D\cup S_T$ (fini), et posons $\fp:=\fP\cap A$. On prétend que $\delta(\fp)\subset \fP$. Commençons par remarquer qu'il existe $l>0$ tel que $(s(\fP))=\fp^l$ en tant qu'idéaux de $A$ : par définition $s(\fP)$ est inversible dans $A_{\ell}$ pour tout $\ell$ différent de $\fp$. En particulier, l'idéal $(s(\fP))$ est une puissance de $\fp$. Si cette puissance était nulle, alors $s(\fP)$ serait une unité de $A$ ce qui est absurde d'après le point \ref{lem:vp-sp>0}. Ainsi, $\delta(s(\fP))=\delta(\fp)^l\subset \fP$ puis, par maximalité de $\fP$, $\delta(\fp)\subset \fP$.

Fixons $a\in A$, et notons $\bar{a}$ son image dans $\bF_\fp\subset \bF_\fP$. Soit $\pi(X)\in \bF_q[X]$ le polynôme minimal de $\bar{a}$ sur $\bF_q$. Comme $\pi(a)\in \fp$, on a $\delta(\pi(a))=\pi^{p^t}(\delta(a))\in \fP$ (cf proposition \ref{prop:propriete-Phia}). Mais comme $\bar{a}^{p^t}$ est aussi racine de $\pi^{p^t}$ modulo $\fP$, il existe $c=c_{\fP,a}\in \{0,...,\deg \pi\}$ tel que 
\begin{equation*}
\delta(a)\equiv a^{p^tq^c} \pmod{\fP}
\end{equation*}
ce qui conclut.
\end{proof}

Il nous reste à prouver que la flèche $\delta$ est une puissance du Frobenius. Pour cela, travaillons dans un cadre  plus général et étudions les triplets $(B,C,f)$ qui vérifient les propriétés $(\operatorname{P}_i)$ suivantes :
\begin{enumerate}[label=$(\operatorname{P}_{\arabic*})$]
    \item\label{prFrobdim} {\it $B\subset C$ sont deux algèbres intègres de type fini sur $\bF_p$ de dimension $1$,}
    \item\label{prFrobcongruence} {\it $f :B\to C$ est un morphisme d'anneaux tel que pour tous les idéaux premiers $\fp$ de $C$, sauf un nombre fini, et tout $b\in B$, il existe $k\in \bN$ (qui peut dépendre de $\fp$ et de $b$), tel que }
\end{enumerate}
\begin{equation}
f(b)\equiv b^{p^k}\pmod{\fp}. \nonumber
\end{equation}
D'après la Proposition \ref{prop:psi-res-frob}, le triplet $(A,\cO_L[S^{-1}],\delta)$ vérifie les propriétés \ref{prFrobdim} et \ref{prFrobcongruence}, et on se ramène donc à prouver :
\medskip
\begin{theo}\label{theo:frob-res}
Soit $(B,C,f)$ vérifiant \ref{prFrobdim} et \ref{prFrobcongruence}. Alors $f$ est la puissance $k$ème du Frobenius où $k$ est un entier possiblement négatif\footnote{\label{foot:k-neg}Notons que dans le cas où $k$ est négatif, l'énoncé entraîne l'existence pour tout $b\in B$ d'une racine $p^{-k}$ dans $C$ (nécessairement unique par intégrité de $C$) qui est l'image de $b$ par la flèche $f$.}.
\end{theo}
\medskip
\begin{rema}
\begin{itemize}
\item L'hypothèse de finitude sur $\bF_p$ de $B$ et $C$ est nécessaire pour assurer que l'on ait assez d'idéaux premiers dans $C$. Par exemple, si $C$ est  local de dimension $1$, la condition \ref{prFrobcongruence} est vide et $f$ peut être n'importe quelle morphisme.
\item Dans le cas particulier qui nous intéresse, à savoir pour le triplet $(A,\cO_L[S^{-1}],\delta)$, s'il existe un élément de $A$ qui n'admet pas de racine $p$èmes dans $L$, cela force l'entier $k$ dans l'énoncé précédent à être positif. Puisque $\delta$ coïncide avec l'élévation à la puissance $p^t$ sur le corps $\bF_q$, on a de plus $k\equiv t\pmod e$.
\end{itemize}
\end{rema}
\medskip
La preuve, qui se déroulera sur les dernières pages de ce texte, revient essentiellement à prouver les deux résultats intermédiaires suivants :
\medskip
\begin{lemm}\label{lem:frob-res-cas-poly}
Soit $(B,C,f)$ vérifiant \ref{prFrobdim} et \ref{prFrobcongruence} où $B$ est une algèbre de polynôme $\bF_p[X]$ et $C=B[f(X)]$. Alors $f(X)=X^{p^k}$ pour un certain entier $k$ possiblement négatif.
\end{lemm}
\medskip
\begin{lemm}\label{lem:unic-k}
Soit $B\subset C$ deux  $\bF_p$-algèbres intègres, et $f:B\to C$ un morphisme d'anneaux tel que pour tout $b\in B$, il existe $k_b\in \bZ$ vérifiant $f(b)=b^{p^{k_b}}$. Alors $f$ est de la forme $\Frob^k$ pour un certain entier $k$ possiblement négatif.
\end{lemm}
\medskip
\begin{rema}
Notons que dans l'énoncé du Lemme \ref{lem:unic-k}, nous n'avons pas supposé que le triplet $(B,C,f)$ vérifie la propriété \ref{prFrobdim}: il n'est pas nécessaire de supposer $B$ ou $C$ de type fini sur $\bF_p$ ou de dimension $1$. 
\end{rema}
\medskip
Terminons d'abord la preuve du Théorème \ref{theo:frob-res} en admettant les deux lemmes.
\begin{proof}[Démonstration du Théorème \ref{theo:frob-res}]
D'après le Lemme \ref{lem:unic-k}, il suffit de montrer que $f(b)$ est de la forme $b^{p^{k_b}}$ pour tout élément $b\in B$, où $k_b\in\bZ$. La flèche $f$ préserve les constantes (i.e. les éléments algébriques sur $\bF_p$) et se restreint donc en un endomorphisme de $\bF_q$ qui est une puissance du Frobenius par théorie de Galois des corps finis. Lorsque $b$ n'est pas une constante, il suffit de montrer que le triplet $(\bF_p[b],\bF_p[b,f(b)],f|_{\bF_p[b]})$ vérifie encore les propriétés \ref{prFrobdim} et \ref{prFrobcongruence} en vertu du Lemme \ref{lem:frob-res-cas-poly}. 

Les algèbres $\bF_p[b]$ et $\bF_p[b,f(b)]$ sont de type fini sur $\bF$ par construction, ainsi qu'intègres en tant que sous algèbres de $C$. $\bF_p[b]$ est de dimension $1$, et il en est de même pour $\bF_p[b,f(b)]$: en effet, on a la chaîne d'inclusions $\bF_p(b)\subset\bF_p(b,f(b))\subset \Frac C$ qui montre que $\bF_p(b,f(b))$ est une extension finie de $\bF_p(b)$. 

Il reste à montrer que la condition \ref{prFrobcongruence} est vérifiée. Par le Lemme \ref{lem:quasi-surj-morph-spec} (énoncé plus bas), tous les idéaux premiers $\fp$ de $\bF_p[b,f(b)]$ excepté un nombre fini sont de la forme $\fP\cap \bF_p[b,f(b)]$ avec un idéal premier $\fP$ de $C$ et vérifie \ref{prFrobcongruence} pour $(B,C,f)$. Ainsi, $f(b)\equiv b^{p^k} \pmod{\fP}$, puis $f(b)-b^{p^k}\in \fP\cap \bF_p[b,f(b)]=\fp$. On en déduit \ref{prFrobcongruence} pour $(\bF_p[b],\bF_p[b,f(b)],f|_{\bF_p[b]})$, ce qui achève la preuve.
\end{proof}

Il nous reste à démontrer les Lemmes \ref{lem:frob-res-cas-poly}, \ref{lem:unic-k}, ainsi que le Lemme \ref{lem:quasi-surj-morph-spec} énoncé plus bas.

\begin{proof}[Démonstration du lemme \ref{lem:frob-res-cas-poly}]
Soit $(B,C,f)$ comme dans l'énoncé. Soit $P(X,Y)\in \bF_p[X,Y]$ un polynôme à deux variables annulateur de $f(X)$ sur $\bF_p(X)$ que l'on suppose de contenu\footnote{On rappelle que le contenu (vu comme un polynôme en $Y$) de $P(X,Y)=\sum_i P_i(X)Y^i$ est le plus grand commun diviseur des $P_i$.} égal à $1$ vu comme un polynôme en $Y$. Le noyau de  la flèche
\[
\bF_p[X,Y]\longrightarrow \bF_p[X,f(X)], \quad Y\mapsto f(X)
\]
est $P(X,Y)\bF_p(X)\cap \bF_p[X,Y]$ qui vaut $P(X,Y)\bF_p[X,Y]$ par hypothèse sur le contenu. Ainsi,
\begin{equation}\label{eq:isoC}
C\cong\bF_p[X,Y]/(P).
\end{equation}

L'identité \eqref{eq:isoC} et la propriété \ref{prFrobcongruence} entraîne l'assertion suivante :

\begin{enumerate}[label=$(A)$]
    \item\label{item:root} Pour tout $x\in \bar{\bF}_p$ sauf un nombre fini, les racines de $P(x,Y)\in \bar{\bF}_p[Y]$ sont toutes de la forme $x^{p^k}$ pour un certain entier $k\in \bZ$;
\end{enumerate}

Il nous reste donc à montrer que  \ref{item:root} entraîne : 
\begin{enumerate}[label=$(B)$]
    \item Il existe un entier $k$ tel que soit $P(X,Y)=X^{p^k}-Y$, soit $P(X,Y)=Y^{p^k}-X$.
\end{enumerate}
Le résultat suivant est une étape cruciale de la preuve:
\medskip
\begin{lemm}\label{lem:xn-res}
Soit $R(X)\in\bF_p(X)$ tel que, pour presque tout $x\in \bar{\bF}_p$, il existe un entier $n_x$  pour lequel $R(x)=x^{n_x}$. Alors il existe  $n\in \bZ$ pour lequel $R(X)=X^n$.
\end{lemm}
\medskip
\begin{proof}
\'Ecrivons $R(X)=R_1(X)/R_2(X)$ pour deux polynômes $R_1(X)$ et $R_2(X)\neq 0$. Soit $\bF$ une extension finie de $\bF_p$ telle que $|\bF|-1$ admet deux diviseurs premiers impairs distincts $\ell_1,\ell_2>2\max (\deg R_1,\deg R_2)$. Pour $\ell\in\{\ell_1,\ell_2\}$, prenons $\zeta_\ell\in \bF\setminus (\bF)^{\ell}$ tel que $\zeta_\ell^{\ell^m}=1$ pour un certain $m$. Soit $x$ une racine $\ell$ème de $\zeta_\ell$ dans $\bar{\bF}$. Quitte \`a agrandir $\bF$, on peut supposer que $R(x)=x^{n_x}$ pour un certain entier $n_x\geq 0$. Comme $X^{\ell}-1$ est scindé à racines simples dans $\bF$ par hypothèse,  $X^\ell-\zeta_\ell$ est le polynôme minimal de $x$ sur $\bF$ et le corps $\bF(x)$ est de degré $\ell$ sur $\bF$ par théorie de Kummer. En écrivant  $n_x=s_x+r_x\ell$ avec $-\ell/2<s_x<\ell/2$, on a $R(x)=\zeta_\ell^{r_x} x^{s_x}$. Quitte à échanger $R_1$ et $R_2$, on peut supposer $0\leq s_x<\ell/2$. Or l'écriture $\zeta_\ell^rx^s\in \bF(x)$ avec $0\le r<\operatorname{ord} \zeta_\ell$ et $0\leq s<\ell/2$ est unique en le couple $(r,s)$ d'où la congruence de polynômes $R_1(X)\equiv \zeta_\ell^{r_x}X^{s_x}R_2(X)\pmod{X^\ell-\zeta_\ell}$, puis $R_1(X)= \zeta_\ell^{r_x}X^{s_x}R_2(X)$ et $R(X)=\zeta_\ell^{r_x}X^{s_x}$ en comparant les degrés. En particulier, $s_x=\deg R$ et $\zeta_\ell^{r_x}$ est indépendant de $\ell\in\{\ell_1,\ell_2\}$. Comme $\zeta_\ell^{r_x}$ est a la fois une racine $\ell_1^m$ et $\ell_2^m$ pour m assez grand, on a $\zeta_\ell^{r_x}=1$ ce qui conclut.
\end{proof}

Écrivons le polynôme $P$ sous la forme $P(X,Y)=Q(X,Y^{p^N})$ avec $N$ maximal pour cette propriété. On observe que la propriété \ref{item:root} pour $P$ entraîne \ref{item:root} pour $Q$. Déterminons $Q$ pour retrouver $P$ ensuite.

Par choix de $N$, nous avons $\partial_Y Q\neq 0$. De plus, $Q$ est irréductible dans $\bF_p(X)[Y]$ car $P$ l'est et on peut trouver par Bézout $S_1(X,Y), S_2(X,Y)\in \bF_p(X)[Y]$ tel que :
\[S_1(X,Y)Q(X,Y)+S_2(X,Y)(\partial_Y Q)(X,Y)=1.\]  On en déduit que, pour presque tout $x\in \bar{\bF}_p$ (plus précisément lorsque $x$ n'est  un pôle d'aucun terme de $S_1(X,Y), S_2(X,Y)$), $Q(x,Y)$ et $(\partial_Y Q)(x,Y)$ sont premiers entre eux et le polynôme $Q(x,Y)$ vu dans $\bar{\bF}_p[Y]$ est scindé à racines simples. Pour de tels $x$, on a alors 
\[
Q(x,Y)=c(x)\prod_{i=1}^d{\left(Y-x^{p^{n_{i,x}}}\right)}
\]
dans $\bar{\bF}_p[Y]$, où $d$ le degré de $Q$ en $Y$ et où $c(X)\in\bF_p[X]$ est le coefficient dominant de $Q(X,Y)$ (resp. $P(X,Y)$) vu comme polynôme en $Y$.  Le terme constant de $(-1)^{\deg_Y Q}Q(x,Y)/c(x)$ est $x^n$ où $n\in\bZ$  est indépendant de $x$ d'après le lemme \ref{lem:xn-res}. \\

Soit $m>d+|n|$ un entier. Quitte à choisir $m$ assez grand, on peut supposer que le groupe cyclique $\bF_{p^m}^\times$ admet un générateur $x$ pour lequel $Q(x,Y)\in \bar{\bF}_p[Y]$ est scindé à racines simples  $x^{p^{k_1}},\dots,x^{p^{k_d}}$ avec  $0\leq k_1<k_2<\cdots <k_d< m$. On obtient la congruence
\[
n\equiv\sum_{i=1}^d{p^{k_i}}\pmod{p^m-1}.
\]
De plus, 
\[
\left|n-\sum_{i=1}^{d}{p^{k_i}}\right|\leq|n|+\sum_{i=1}^{d}{p^{k_i}}<\sum_{k=0}^m {p^{k}}\le p^m-1,
\] car $d+|n|<m$.
Ainsi, $n=\sum_i {p^{k_i}}$.

Par unicité de la l'écriture en base $p$, on voit que les $(k_i)_i$ ne dépendent pas du choix\footnote{Ici, on utilise de manière cruciale le fait que les $k_i$ sont distincts. Par exemple, l'égalité $p^2=p+\dots +p$ fournit un contre-exemple lorsque certains exposants sont répétés.} de l'élément $x$ vérifiant les  contraintes  précédentes. En particulier, puisque les polynômes $Q(X,Y)$ et $c(X)\prod_i (Y-X^{p^{k_i}})$ en $X$ sur $\bF_p[Y]$ coïncident pour une infinité de valeurs de $X$ dans $\bar{\bF}_p$, ces polynômes sont égaux. Par irréductibilité de $Q$ et comme $\partial_YQ\neq 0$, on obtient $c(X)=1$ et $d=1$, d'où $Q(X,Y)=Y-X^{p^{k_1}}$ puis $P(X,Y)=Y^{p^N}-X^{p^{k_1}}$. Enfin, l'irréductibilité de $P$ impose $N=1$ ou $k_1=1$, ce qui est équivalent à
\[
P(X,Y)=X^{p^k}-Y\ {\rm ou}\ P(X,Y)=Y^{p^k}-X
\]
et la preuve est terminée !
\end{proof}

\begin{proof}[Démonstration de \ref{lem:unic-k}]
Soient $x,y\in B$, on veut montrer que\footnote{Ici, nous avons réalisé un léger abus car les entiers ne sont pas forcément uniques. Nous entendons plutôt avoir les égalités $f(x)=x^{p^{k_x}}=x^{p^{k_y}}$ (idem pour $f(y)$} $k_x=k_y$. On raisonne alors sur la sous-$\bF_p$-algèbre $\bF_p[x,y]$ de $B$ engendrée par $x$ et $y$ où la restriction de $f$ vérifie encore la propriété de l'énoncé. De même, il suffit de prouver le résultat lorsque l'on remplace $f$ par $f\circ\Frob^m$ avec $m$ assez grand et on peut donc supposer $k_x,k_y\ge 0$ et  $f$ est alors un endomorphisme de $\bF_p[x,y]$. On distingue les trois cas suivants qui sont exhaustifs par intégrité de $\bF_p[x,y]$:
\begin{enumerate}
\item Les éléments $x$ et $y$ sont algébriques sur $\bF_p$. Alors $\bF_p[x,y]$ est un corps fini et $f$ est une puissance du Frobenius par théorie de Galois sur $\bF_p$. 

\item Si  $y$ est transcendant sur $\bF_p(x)$, écrivons
\[
x^{p^{k_{xy}}}y^{p^{k_{xy}}}=f(xy)=f(x)f(y)=x^{p^{k_{x}}}y^{p^{k_{y}}}.
\]
d'où $k_y=k_{xy}$ et $x^{p^{k_{x}}}=x^{p^{k_{xy}}}=x^{p^{k_y}}$
\item Si  $x$ et $y$ sont transcendants sur $\bF_p$ et mais $y$ est algébrique sur $\bF_p(x)$.

Pour tout idéal $\fb$ de $B$, on a $f(\fb)\subset \fb$ par hypothèse. Ainsi, pour tout idéal maximal $\fm\subset B$ de corps résiduel $\kappa_\fm$, $f$ induit un endomorphisme $\bar{f}$ de $\kappa_\fm$. Notons $d_\fm$ le degré de $\fm$, i.e. la dimension de $\kappa_\fm$ sur $\bF_p$, ainsi que $n_{\bar{x}}$ et $n_{\bar{y}}$ les dimensions sur $\bF_p$ des sous-corps de $\kappa_\fm$ engendrés respectivement par $\bar{x}$ et $\bar{y}$. Puisque $\kappa_\fm=\bF_p[\bar{x},\bar{y}]$, on a $d_\fm=[n_{\bar{x}},n_{\bar{y}}]$.

Puisque $\kappa_\fm$ est un corps fini et $\bar{f}$ un endomorphisme de $\kappa_\fm$, il existe un entier $0\leq k_\fm<d_\fm$ tel que $\bar{f}(\bar{b})=\bar{b}^{k_\fm}$ pour tout $b\in B$. En particulier, $k_x\equiv k_\fm \pmod{n_{\bar{x}}}$ et $k_y\equiv k_\fm\pmod{n_{\bar{y}}}$. On obtient $k_x\equiv k_y\pmod{(n_{\bar{x}},n_{\bar{y}})}$. Pour conclure, il suffit de montrer que le pgcd $(n_{\bar{x}},n_{\bar{y}})$ peut être choisi arbitrairement grand (en faisant varier $\fm$).

Soit $P_x(Y)\in \bF_p[x][Y]$ le polynôme minimal de $y$ sur $\bF_p(x)$ (normalisé). Choisir un couple $(\bar{x},\bar{y})\in \bar{\bF}_p^2$ tel que $P_{\bar{x}}(\bar{y})=0$ produit un idéal maximal \[\fm_{(\bar{x},\bar{y})}=\ker(\bF_p[x,y]\to \bar{\bF}_p:(x,y)\mapsto (\bar{x},\bar{y})),\] et tous les idéaux maximaux de $\bF_p[x,y]$ sont de cette forme. 

Soit $d$ le degré de $P_x(Y)$ en $Y$. Soit $\fm$ un idéal maximal de $B$ provenant d'un point géométrique $(\bar{x},\bar{y})$. Puisque $\bar{y}$ est racine de $P_{\bar{x}}(Y)$, on a $d_\fm\leq d\cdot n_{\bar{x}}$. Alors, on a l'inégalité
\[
(n_{\bar{x}},n_{\bar{y}})=\frac{n_{\bar{x}}n_{\bar{y}}}{d_\fm}\geq \frac{n_{\bar{y}}}{d}.
\]
D'après le Lemme \ref{lem:quasi-surj-morph-spec}, le complémentaire de l'image de l'application $\operatorname{Spm}\bF_p[x,y] \rightarrow \operatorname{Spm} \bF[y] $ est un ensemble fini et on peut choisir $(\bar{x},\bar{y})$ de telle sorte que $n_{\bar{y}}$ soit arbitrairement grand, ce qui conclut.  
\end{enumerate}
\end{proof}

Les preuves du théorème et du lemme ont utilisé l'énoncé classique suivant :
\medskip
\begin{lemm}\label{lem:quasi-surj-morph-spec}
Soit $B\subset C$ deux algèbres intègres  de type fini sur $k$ un corps de dimension $1$, tous les idéaux premiers de $B$ sauf un nombre fini sont de la forme $\fP \cap B$ avec $\fP$ un idéal premier de $C$.
\end{lemm}
\medskip
\begin{proof}
Comme les fermés $V(\lambda)$ sont des ensembles finis dans $|\Spec B|$ pour tout élément $\lambda$ non inversible dans $B$ par dimension, il suffit de trouver un élément $\lambda$ pour lequel la flèche naturelle $|\Spec C[1/\lambda]|\to |\Spec B[1/\lambda]|$ est surjective. Pour trouver cet élément, on observe que $C$ est de type fini sur $B$ et on choisit une famille finie$(x_i)_i$ de générateurs. Comme toutes les algèbres sont de dimension $1$, on peut trouver un polynôme annulateur sur $B$ pour chacun de ces générateurs dont on note $\lambda_i$ le coefficient dominant. Par construction, l'inclusion $B[1/\prod_i \lambda_i]\to C[1/\prod_i \lambda_i]$ est finie et induit une flèche surjective au niveau des spectres,  l'élément $\lambda=\prod_i \lambda_i$ convient.
\end{proof}

\end{document}